\documentclass{article}
\usepackage{fullpage}
\usepackage{amssymb}
\usepackage{amsthm}
\usepackage{amsmath}
\usepackage{graphicx}
\usepackage{subcaption}
\usepackage{color}
\newcommand{\uu}{\mathbf u}
\newcommand{\ff}{\mathbf f}
\newcommand{\inn}{\textrm{ in }}
\newcommand{\on}{\textrm{ on }}
\newcommand{\bg}{\mathbf g}
\newcommand{\n}{\mathbf n}
\newcommand{\xx}{\mathbf x}
\newcommand{\vv}{\mathbf v}
\newcommand{\VV}{\mathbf V}
\newcommand{\WW}{\mathbf W}
\newcommand{\ww}{\mathbf w}
\newcommand{\TT}{\mathbf E_{sm}}
\newcommand{\T}{E_{sm}}
\newcommand{\GG}{\mathbf G_h}
\newcommand{\osigma}{\overline\sigma}
\newcommand{\tensor}[1]{\overline{\overline{#1}}}
\newcommand{\sci}[1]{\times 10^{#1}}
\newcommand{\zz}{\mathbf z}

\newcommand{\qq}{\mathbf q}
\newtheorem{proposition}{Proposition}
\newtheorem{algorithm}{Algorithm}
\newtheorem*{remark}{Remark}
\title{A Coupling Approach for Linear Elasticity Problems with Spatially Noncoincident Interfaces}
\author{Pavel Bochev, James Cheung,  Max Gunzburger, and Mauro Perego}
\date{}
\begin{document}
	\maketitle
	\begin{abstract}
		We present a new formulation based on the classical Dirichlet--Neumann
		formulation for interface coupling problems in linearized elasticity. 
		By using Taylor series expansions, we derive a new set of interface
		conditions that allow our formulation to pass the linear consistency test.
		In addition, we propose an iterative method to determine the solution
		of our formulation. We demonstrate in our numerical results that
		we may achieve the desired piecewise linear finite element error bounds for
		both nonoverlapping domain decomposition problems as well as for
		interface coupling problems where the Lam\'e parameters of the structures
		differ.
	\end{abstract}
	\section{Introduction}
	In many important physical applications, two systems are coupled together 
	through a physical interface. It is on this interface that physical quantities,
	such as stresses, are transferred between the two systems. Mathematically, 
	such systems are modeled by two sets of equations with a set of additional conditions
	posed at the interface. We refer to these formulations as \emph{interface coupling}
	formulations. If the equations are the same on both sides of this interface,
	the interface coupling formulations fall under a class of problems called
	\emph{non--overlapping domain decomposition} formulations. 
	
	In the continuous setting, the interface exists only as a single curve that lies
	between the two problem subdomains; however, in the finite element setting,
	this is generally not the case. Often, the two subdomains of the interface
	coupling problem are meshed separately, and the single interface in the
	continuous setting becomes two distinct interfaces. If the positions of the
	nodes in the two interfaces are matching, then there exists no problem
	in determining the solution of the discretized problem. Most often,
	classical methods for determining solutions to interface coupling problems
	in the continuous setting
	may be extended to the finite element formulation in this case. A thorough
	exposition of classical interface coupling methods may be found in 
	\cite{quarteroni1999domain} and \cite{toselli2005domain}.
	
	If the two interfaces
	are spatially coincident with nonmatching nodes, one may introduce various
	operators that map values from one interface to the other without much difficulty.
	In fact, there exists an entire body of literature that discusses solution methods
	for this kind of discrete interface problem. Among the methods that deal with 
	this kind of nonmatching interface problem, one of the most studied and utilized
	methods are mortar element methods \cite{bernardi1994new,
	wohlmuth2000mortar}. 

	 However, if the continuous interface
	is a curved surface, the interfaces generated by separate meshing will often
	not spatially coincide. The most prominent issue in this setting is that the interface
	conditions defined in the continuous setting has little meaning
	since there is no clear definition of an interface condition when the discrete 
	interfaces are spatially noncoincident. When comparing the volume of literature
	available for solution methods of the case of spatially noncoincident interface
	coupling, one will find that the amount of available literature on this subject
	is quite sparse. Most often, the methods discussed in the previous paragraph
	are generalized to this case.

%
	A desirable property of any discrete approximation to its continuous counterpart
	is that it is consistent. This means that the behavior of the discrete model should
	match the behavior of the continuous model. A common approach to test whether
	a discrete formulation is consistent is whether or not it can reproduce $k$--degree
	polynomials, where $k$ is the polynomial degree used in the approximation space.
	As of current, no method in the literature is $k$--degree consistent; however, there
	are a few that are linearly consistent. A novel approach presented in
	\cite{laursen2003consistent} uses an energy correction
	approach to remove excess energy that arises from the overlaps generated from
	the spatially noncoincident interfaces and adds additional energy to the deficit
	incurred by the gaps between the noncoincident interfaces. The complexity
	of the implementation of this method may be prohibitive since a mesh--like structure
	must be constructed between the discrete interfaces to link the two interfaces. 
	The method proposed in \cite{parks2007novel}
	removes this complexity by perturbing the meshes at the discrete interfaces so that
	the area of the gaps and overlaps in the subdomain meshes are equal. 
	\cite{day2008analysis} removes this perturbation requirement by requiring that the subdomain
	meshes are only overlapping. Because these methods are based on minimizing a variational
	principle over the entire problem domain, their applicability is limited to cases
	where the material constant on both sides of the domain are the same.

	In this work, we present an interface coupling approach for spatially noncoincident
	interfaces for the case of coupled linearized elasticity equations. Through the use
	of Taylor series expansions, we introduce a new set of interface conditions 
	based on the classical Dirichlet--Neumann formulation that 
	allow us to pass the linear patch test in the domain decomposition problem. 
	Subsequently, we introduce an iterative solution method that allows us to 
	solve this formulation. The iterative method is particularly useful if 
	a coupled solution must be determined from separate codes. 
	To ensure that we achieve the expected accuracy from 
	piecewise linear finite element methods, we utilize the Zhang--Naga gradient
	recovery operator \cite{zhang2005new} to recover a second order accurate stress tensor as Neumann data
	in the coupled formulation. 
	An 
	additional benefit of this new formulation is that, unlike the methods described 
	in the previous paragraph, this method requires no intermediate meshing procedure,
	nor does it require any sort of mesh perturbation in the neighborhood of the
	discrete interfaces. In addition, this coupling approach is also applicable to
	interface coupling problems where the Lam\'e parameters differ between
	the coupled subdomains. 
	
	The remainder of this paper will be structured as follows. In Sect.
	\ref{problem statement} we establish the framework of our problem and 
	define our monolithic coupling formulation. In Sect. \ref{iterative solution method},
	we introduce an iterative solution method for the coupling formulation. In 
	Sect. \ref{Numerical Results}, we will demonstrate that our coupling 
	formulation achieves the desired finite element error bounds for both 
	interface coupling and domain decomposition problems. And finally,
	in Sect. \ref{conclusion} we will provide some concluding statements
	and provide some insight in the future directions of this work. An appendix
	is provided at the end of this paper to give the reader insight on the 
	gradient recovery method utilized in this work.
	
	\section{Statement of the Interface Problem} \label{problem statement}
	Assume that $\Omega$ is a domain; i.e., a simply connected bounded subset of $\mathbb R^2$
	with a Lipschitz continuous boundary.  We shall denote the boundary of $\Omega$ as
	$\partial \Omega$. In addition, let $\uu$ denote the
	virtual displacement of a material characterized by its Lam\'e parameters
	$\lambda$ and $\mu$. Then, using standard notation, we define
	\begin{equation*}
		\varepsilon(\uu) = \frac12 \left( \nabla \uu + (\nabla \uu)^T \right)
	\end{equation*}
	as the strain tensor, and 
	\begin{equation*}
		\sigma(\uu) = \lambda \textrm{ tr }\varepsilon(\uu) \mathbf I + 2 \mu \varepsilon(\uu)
	\end{equation*}
	as the stress tensor, where $\textrm{tr}(\cdot)$ denotes the matrix trace operator. 
	Then, the statement of linearized elasticity is given by
	\begin{equation*}
		\begin{aligned}
			- \nabla \cdot \sigma(\uu) &= \ff \inn \Omega \\
							\uu &= \bg \on \partial \Omega,
		\end{aligned}
	\end{equation*}
	where $\ff$ is the external force applied to the structure, and $\bg$ is the
	initial displacement of the structure on its boundary. 
	
	\begin{figure}
		\centering
		\resizebox{0.40\textwidth}{!}{
		\input{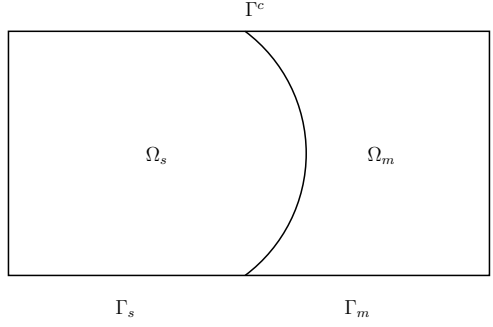}
		}
		\caption{A sketch of the continuous interface coupling domain}
	\end{figure}
	\subsection{The Continuous Interface Problem}
	Let us now assume that there exists two domains, $\Omega_i$, $i=s,m$,
	such that the intersection of their boundaries is a simply connected curve. Let
	$\Gamma^c := \partial \Omega_s \cap \partial \Omega_m$ be the \emph{interface}
	between $\Omega_i$, and $\Gamma_i := \partial \Omega_i \setminus \Gamma^c$
	be the complement of the interface on the boundary $\partial \Omega_i$.
	In addition, let $\n_i$ be the outward unit normal vector of $\partial \Omega_i$.
	Let us denote $\uu_i$ as the virtual displacement over $\Omega_i$ for a material characterized by 
	its Lam\'e parameters $\lambda_i$ and $\mu_i$. Assume that there is no slip
	at $\Gamma^c$; then, virtual displacements $\uu_i$ is governed by the system
		\begin{subequations} \label{interface problem}
		\begin{equation}
		\begin{aligned}
			-\nabla \cdot \sigma_i(\uu_i) &= \ff_i \inn \Omega_i \\
							\uu_i &= \bg_i \on \Gamma_i
		\end{aligned}
		\end{equation}
		\textrm{ subject to }\quad
		\begin{equation} \label{interface conditions}
		\begin{aligned}
			\uu_s - \uu_m &=0 \\
			\sigma_s(\uu_s)\n_m - \sigma_m(\uu_m)\n_m &= 0
		\end{aligned}
		\quad\textrm{on } \Gamma^c,
		\end{equation}
		\end{subequations}
	 where $\sigma_i(\uu_i) = \lambda_i (\nabla \cdot \uu_i) \mathbf I + 2 \mu_i \varepsilon(\uu_i) $,
	 $\ff_i$ is the external force applied to each structure, and $\bg_i$ is the initial displacement
	 of each structure on $\Gamma_i$. The interface conditions defined on $\Gamma^c$
	 in \eqref{interface conditions} describes the transmission of stresses between the two 
	 structures. If $\lambda_s = \lambda_m$ and $\mu_s = \mu_m$, then \eqref{interface problem}
	 may be referred to as a nonoverlapping \emph{domain decomposition} problem because it models
	 the same structural material over both subdomains.
	 
	Using the standard Sobolev space notation, let
	\begin{equation*}
		H^1_{\Gamma_i}(\Omega_i) = \left\{ v \in H^1(\Omega_i); v = 0 \on \Gamma_i \right\},
	\end{equation*}
	\begin{equation*}
		\VV_i = \left( H^1_{\Gamma_i}(\Omega_i) \right)^2,
	\end{equation*}
	and
	\begin{equation*}
		\VV_{i,0} = \left(H^1_0(\Omega_i)\right)^2
	\end{equation*}
	where, again, $i=s,m$.
	In addition, we shall define the interface trace space
	\begin{equation*}
		\WW = \left( H^{1/2}_{00}(\Gamma^c) \right)^2.
	\end{equation*}
	
	In this work, we consider \eqref{interface problem} in the weak form:
	 seek $(\uu_s, \uu_m)\in \mathbf V_s \times \mathbf V_m$ such that
	\begin{subequations} \label{galerkin form}
	\begin{equation} \label{dirichlet problem}
		\begin{aligned}
			\int_{\Omega_s} \sigma_s(\uu_s):\epsilon(\vv_s) dx 
				&= \int_{\Omega_s} \ff_s \cdot \vv_s dx
				&\quad \forall \vv_s \in \VV_{s,0} \\
			\int_{\Gamma^c}\left(\uu_s - \uu_m \right) \cdot \ww ds &= 0
				&\quad \forall \ww \in \WW
		\end{aligned}
	\end{equation}
	\begin{equation} \label{neumann problem}
		\int_{\Omega_m} \sigma_m(\uu_m):\epsilon(\vv_m)dx
		- \int_{\Gamma^c}\left(\sigma_s(\uu_s) \n_m\right) \cdot \vv_m ds
		= \int_{\Omega_m} \ff_m \cdot \vv_m dx
		\quad \forall \vv_m \in \VV_m.
	\end{equation}
	\end{subequations}
	 The set of equations \eqref{galerkin form} is the
	\emph{Dirichlet--Neumann} coupling of the linearized elasticity
	equations, since the matching of the Dirichlet conditions on $\Gamma^c$
	is weakly enforced in \eqref{dirichlet problem} and the matching of the
	Neumann conditions on $\Gamma^c$ are enforced in 
	\eqref{neumann problem}.
	
	\subsection{Finite Element Approximation}
	Let $\Omega_{i, h}$, $i=s,m$, be the triangulation of $\Omega_i$, with mesh size
	$h_i$ and $\partial \Omega_{i,h}$ be their boundaries. Assume for now that
	$\partial \Omega_{s,h} \cap \partial \Omega_{m,h}$ is a simply connected curve;
	i.e., the interface between the triangulations $\Omega_{s,h}$ and $\Omega_{m,h}$
	is \emph{spatially coincident}. Let us denote this discretized interface $\Gamma^c_h$
	and let $\Gamma_{i,h} := \partial \Omega_{i,h} \setminus \Gamma^c_h$. 
	Also, let us define
	\begin{equation*}
		V^h_i \subset H^1_{\Gamma_i}(\Omega_{i,h})
	\end{equation*}
	\begin{equation*}
		V^h_{i,0} = \left\{ v \in V^h_i; v = 0 \on \partial \Omega_{i,h} \right\},
	\end{equation*}
	and 
	\begin{equation*}
		W^h \subset H^{1/2}_{00}(\Gamma^c_h)
	\end{equation*}
	as standard piecewise linear approximation spaces. In addition, let
	$\VV^h_i := \left(V^h_i\right)^2$, $\VV^h_{i,0}:= \left(V^h_{i,0}\right)^2$ and 
	$\WW^h := \left( W^h \right)^2$ be the finite dimensional product spaces. 
	The statement of the finite element
	discretization of \eqref{galerkin form} becomes the following:
	seek $(\uu_s^h, \uu_m^h) \in \VV^h_s \times \VV^h_m$ such that
	\begin{subequations} \label{fem galerkin form}
	\begin{equation}
		\begin{aligned}
		\int_{\Omega_{s,h}} \sigma_s(\uu_s^h):\epsilon(\vv_s^h)dx
			&= \int_{\Omega_{s,h}}{\ff_s \cdot \vv_s} dx &\quad \forall \vv^h_s \in \VV^h_{s,0} \\
		\int_{\Gamma^c_h} \left(  \uu_s^h - \uu_m^h \right)\cdot \ww^h ds &= 0
			&\quad \forall \ww^h \in \WW^h
		\end{aligned}
	\end{equation}
	\begin{equation}
		\int_{\Omega_{m,h}} \sigma_m(\uu^h_m): \epsilon(\vv^h_m) 
		- \int_{\Gamma^c_h} \left( \sigma_s(\uu^h_s)\n_m \right) \cdot \vv^h_m ds
		= \int_{\Omega_{m,h}} \ff_m \cdot \vv_m dx
		\quad \forall \vv_m^h \in \VV^h_m,
	\end{equation}
	\end{subequations}	
	where $\uu^h_i$, $i=s, m$, denotes the discretized virtual displacement.
	
	\subsubsection{Spatially Noncoincident Interfaces}
	\begin{figure}
		\centering
		\resizebox{0.4\textwidth}{!}{
		\input{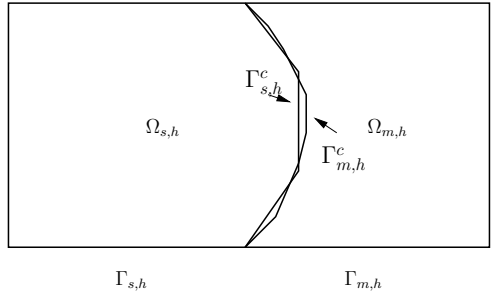}
		}
		\caption{An illustration of subdomains where their interfaces are spatially noncoincident.}
	\end{figure}
	We shall now do away with the assumption that the interfaces between $\Omega_{s,h}$
	and $\Omega_{m,h}$ are spatially coincident. Let $\Gamma^c_{i,h}$, $i=s, m$,
	be the discretization of $\Gamma^c$ that belongs to $\partial \Omega_{i,h}$ and
	$\Gamma_{i,h} := \partial \Omega_{i,h} \setminus \Gamma^c_{i,h}$. In addition, let
	$V_i^h$ and $V_{i,0}^h$ again be the the piecewise linear approximation spaces
	over $\Omega_{i,h}$ redefined with the change of definition of $\Gamma_{i,h}$.
	Also again, let $\VV^h_i = \left( V_i^h \right)^2$ and $\VV^h_{i,0} = \left( V_{i,0}^h \right)^2$.
	Now let us define 
	\begin{equation*}
		W^h_{i} \subset H^{1/2}_{00}(\Gamma^c_{i,h})
	\end{equation*}
	as the piecewise linear approximation space over $\Gamma^c_{h,i}$
	and $\WW^h_i:= \left( W^h_i\right)^2$ as the
	product space over the trace.
	
	The difficulty in coupling problems with noncoincident discrete interfaces presents
	itself in the interface conditions. Because $\uu_s \big|_{\Gamma^c_{s,h}}$ and
	$\uu_m \big|_{\Gamma^c_{m,h}}$ are defined over separate interfaces, it becomes
	a troublesome ordeal to define a suitable set of interface conditions. In this work,
	we deal with this issue by extending the values related to $\uu_s\big|_{\Gamma^c_{s,h}}$
	to $\Gamma^c_{m,h}$ by using a Taylor series expansion.

	\paragraph{Taylor Series Expansion Operators}
	Before defining our discrete interface conditions we must first define some preliminary material.
	 Let $\xx_s$ and $\xx_m$ be points in $\Gamma^c_{s,h}$
	and $\Gamma^c_{m,h}$ respectively, then defining 
	\begin{equation*}
	h := \max_{i = s,m} h_i,
	\end{equation*} we 
	assume that there exists mappings 
	\begin{equation*}
	\phi_{sm}: \xx_s \in \Gamma^c_{s,h}\rightarrow \xx_m \in \Gamma^c_{m,h} 
	\quad \textrm{and} \quad
	 \phi_{ms}: \xx_m \in \Gamma^c_{m,h}\rightarrow
	\xx_s \in \Gamma^c_{s,h}
	\end{equation*}
	such that  
	\begin{equation} \label{distance condition}
		\sup_{\xx_s\in \Gamma^c_{s,h}} d(\xx_s, \phi_{sm}(\xx_s))\leq C h
	\quad\textrm{and}\quad
		\sup_{\xx_m \in \Gamma^c_{m,h}} d(\xx_m, \phi_{ms}(\xx_m)) \leq C h,
	\end{equation} where $C$ is an arbitrary constant. To simplify notation,
	we set $\hat\xx_m := \phi_{sm}(\xx_s)$ and $\hat\xx_s := \phi_{ms}(\xx_m)$,
	where \textbf{the dependence of $\hat\xx_m$ on $\xx_s$ and $\hat\xx_s$ on $\xx_m$ 
	is implied.}
		
	We are now ready to define our Taylor expansion operators. Let 
	$\T: \VV^h_s \rightarrow \WW^h_s$ be defined as
	\begin{equation} \label{taylor displacement}
		\T \vv(\xx_s) = \vv(\xx_s) + \nabla \vv(\xx_s) \left(\hat\xx_m - \xx_s\right)
	\end{equation}
	and $\TT: (\VV_s^h)^2 \rightarrow (\VV_m^h)^2$ be defined as
	\begin{equation} \label{taylor jacobian}
		\TT\tensor{v}(\xx_m) = \tensor{v}(\hat\xx_s)+ 
								\left[ \partial_x \tensor{v}(\hat\xx_s)
									 \left( \xx_m - \hat\xx_s \right), 
								\ \partial_y \tensor{v}(\hat\xx_s)
								 \left( \xx_m - \hat\xx_s\right) \right],
	\end{equation}
	where $\tensor{v} \in \left(\VV^h_s\right)^2$ is a rank two tensor . 
	\emph{
	We emphasize that $\T(\cdot)$ maps functions defined over $\Gamma^c_{s,h}$ to
	functions defined over $\Gamma^c_{s,h}$, since $\hat \xx_m$ depends on $\xx_s$. On the contrary, 
	$\TT(\cdot)$ maps functions defined over $\Gamma^c_{s,h}$ to functions defined over
	$\Gamma^c_{m,h}$, because $\hat\xx_s$ depends on $\xx_m$.}
	
	\paragraph{A Noncoincident Interface Formulation}
	Central to our method's
	ability to achieve second order accuracy in the $L^2$--norm is the
	superconvergent gradient recovery operator. There are various methods
	in the literature that allow us to recover a superconvergent gradient; i.e.,
	see \cite{zienkiewicz1992superconvergent, zienkiewiczsuperconvergent, 		
		levine1985superconvergent}. In this work, we choose
	to use the Zhang--Naga gradient recovery operator \cite{zhang2005new, naga2004posteriori}. 
	We shall denote $G_h: V^h_s \rightarrow \VV^h_s$
	as the recovered gradient.

	Borrowing terminology from the mortar element method, we call $\Gamma^c_{s,h}$
	the \emph{slave} interface and $\Gamma^c_{m,h}$ the \emph{master} interface. It is
	on the master interface, $\Gamma^c_{m,h}$, that we wish to enforce some sort of continuity of the 
	virtual displacements and the normal stresses. 
	Let the extended strain tensor, $\overline \varepsilon: \VV^h_{s} \rightarrow \left(\VV^h_m\right)^2$,
 	be defined as
	\begin{equation*}
		\overline \varepsilon (\vv) := \frac12 \left\{ \TT \left(\GG \vv\right) + \left[\TT \left(\GG \vv\right)\right]^T \right\},
	\end{equation*}
	where
	\begin{equation*}
		\GG \left[\begin{array}{c} v_1 \\ v_2 \end{array} \right] = \left[\begin{array}{c}
					\left(G_h v_1\right)^T \\
					\left(G_h v_2\right)^T
				\end{array}\right]
	\end{equation*}
	is the superconvergent Jacobian constructed from the recovered gradient of the vector components
	of $\vv$.
	Using the Taylor expansion operators
	defined in \eqref{taylor displacement} and \eqref{taylor jacobian}, we want to 
	enforce
	\begin{subequations} \label{coupling conditions}
	\begin{equation} \label{dirichlet condition}
		\T \uu_s^h = \uu_m^h(\hat\xx_m) \on \Gamma^c_{s,h},
	\end{equation}
	\textrm{and} 
	\begin{equation} \label{neumann condition}
		\sigma_m(\uu_m^h)\n_m = \overline \sigma_s(\uu_s^h) \n_m  \on \Gamma^c_{m,h},
	\end{equation}
	\end{subequations}
	where $\osigma_s(\uu_s^h) = \lambda_s \textrm{ tr}\left( \overline \varepsilon(\uu_s^h)\right)
	+ 2\mu_s \overline \varepsilon(\uu_s^h)$ is the stress tensor constructed from the 
	extended strain tensor. Again, note that \eqref{dirichlet condition} is
	 an equation posed over $\Gamma^c_{s,h}$ despite $\uu^h_m$ being
	defined over $\Gamma^c_{m,h}$ due to the dependence of $\hat\xx_m$ on $\xx_s$. 
	Also note that \eqref{neumann condition} 
	is an equation posed over $\Gamma^c_{m,h}$. We provide a 1D interpretation 
	of these interface conditions in Fig. \ref{1D illustration} to illustrate the meaning 
	of these conditions. 
	
	\begin{figure}
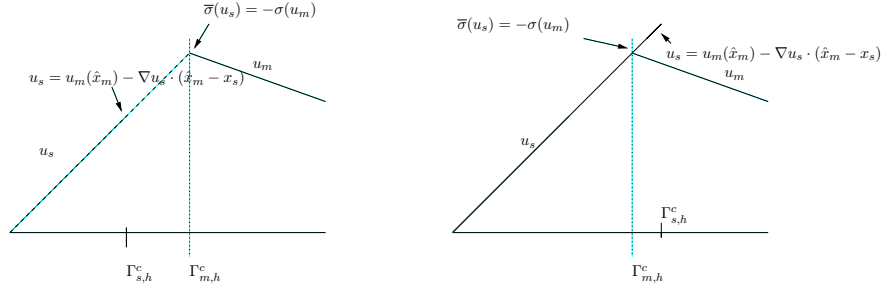

		\centering
		\resizebox{0.2\textheight}{!}{
		\input{./figures/1dgap_t}
		}
		\qquad\qquad
		\resizebox{0.2\textheight}{!}{
		\input{./figures/1doverlap_t}
		}
		\caption{Illustrations of the interface conditions 
			\eqref{coupling conditions}. Left: Gap at the interface. The dashed line represents
			the linear extrapolation of $u_s$. Right: Overlapping subdomains.}
		\label{1D illustration}
	\end{figure}
	
	\begin{remark}
	In this work, we make the fundamental assumption that
	the master interface, $\Gamma^c_{m,h}$, is an $\mathcal O(h^2_m)$ approximation of the continuous 
	interface $\Gamma^c$ so that we can achieve optimal convergence rates in the following coupling
	approach. Otherwise, the error in the interface approximation will dominate the error of the
	polynomial approximation.
	\end{remark}
	
	With the interface conditions
	\eqref{dirichlet condition} and \eqref{neumann condition} defined, we may 
	pose the noncoincident interface coupling problem as the following:
	seek $(\uu_s^h, \uu_m^h) \in \VV^h_s \times \VV^h_m$ such that
	\begin{subequations}  \label{discrete problem}
		\begin{equation}
			\begin{aligned} \label{not dirichlet}
			\int_{\Omega_{s,h}} \sigma_s(\uu_s^h):\varepsilon(\vv_s^h) dx
			&= \int_{\Omega_{s,h}} \ff_s \cdot \vv_s^h dx 
			&\quad \forall \vv_s^h \in \VV^h_{s,0} \\
			\int_{\Gamma^c_{s,m}} \left( \T \uu_s^h - \uu_m^h\left(\hat\xx_m \right) \right)
				\cdot \ww^h_s ds &= 0
			&\quad \forall \ww^h_s \in \WW^h_s
			\end{aligned}
		\end{equation}
		\begin{equation} \label{neumann}
			\int_{\Omega_{m,h}} \sigma_m(\uu_m^h):\varepsilon(\vv_m^h) dx
			- \int_{\Gamma^c_{m,h}} \left( \osigma(\uu_s^h)\n_m \right)\cdot \vv_m^h ds
			= \int_{\Omega_{m,h}} \ff_m \cdot \vv_m^h dx
			\quad \forall \vv_m^h \in \VV^h_m.
		\end{equation}
	\end{subequations}
	In the remainder of the paper, we shall use a more compact notation to describe
	\eqref{discrete problem}. Let the bilinear form $a^h_i: \VV_i^h \times \VV_i^h$, $i=s, m$, be
	defined as
	\begin{equation*}
		a^h_i(\vv, \ww) = \int_{\Omega_{i,h}} \sigma_i(\vv): \varepsilon(\ww) dx,
	\end{equation*}
	then \eqref{discrete problem} can be represented as: Seek
	$(\uu_s^h, \uu_m^h) \in \VV_s^h \times \VV^h_m$ such that
		\begin{equation} \label{compact discrete}
			\begin{aligned} 
			a^h_s(\uu_s^h, \vv_s^h) &= \left< \ff_s, \vv_s\right>_{\Omega_{s,h}}
			&\quad \forall \vv_s^h \in \VV^h_{s,0} \\
			\left<\T \uu_s^h - \uu_m^h\left(\hat\xx_m \right), 
				\ww^h_s\right>_{\Gamma^c_{s,h}} &= 0
			&\quad \forall \ww^h_s \in \WW^h_s \\
			a^h_m(\uu_m^h, \vv_m^h) 
			- \left< \osigma(\uu_s^h)\n_m, \vv_m^h \right>_{\Gamma^c_{m,h}}
			&= \left< \ff_m, \vv_m^h\right>_{\Omega_{m,h}}
			&\quad \forall \vv_m^h \in \VV^h_m,
			\end{aligned}
		\end{equation}
	where $\left<\cdot, \cdot \right>_D$ is the standard duality pairing with the integral
	taken over $D$, where $D$ can be $\Omega_{i,h}$ or $\Gamma^c_{i,h}$. The
	``Dirichlet'' subproblem in the above formulation is not a Dirichlet condition in the
	classical sense; it is perturbed by a oblique derivative term introduced by the
	Taylor series expansion. The interface condition used in the 
	Neumann subproblem remains a classical natural boundary condition on 
	the interface $\Gamma^c_{m,h}$. 

	\paragraph{Linear Consistency}
	As stated in the introduction, linear consistency is an important property for 
	a noncoincident coupling method to possess. In the following proposition, we show that
	\eqref{compact discrete} unconditionally possesses linear consistency.
	\begin{proposition}
		Assume that $(\lambda_s, \mu_s) = (\lambda_m, \mu_m)$ and $(\ff_s, \ff_m) = (0,0)$. 
		If $\qq_m$ and $\qq_s$ are linear vector functions with the same analytic 
		representation, then $(\qq_s^h, \qq_m^h)$ is a solution of \eqref{compact discrete}.
	\end{proposition}
	\begin{proof}
		Let $\mathcal R_h$ be any continuous lifting operator that maps
		$\WW^h_m$ into $\VV^h_m$, then we may decompose any function
		$\vv \in \VV^h_m$ into the following form
		\begin{equation} \label{H1 decomposition}
			\vv = \vv^0 + \mathcal R_h \vv \big|_{\Gamma^c_{m,h}},
		\end{equation}
		where $\vv^0 \in \VV^h_{m,0}$ and $\vv\big|_{\Gamma^c_{m,h}} \in \WW^h_m$. From this,
		we may represent \eqref{compact discrete} in the following form:
		Seek $\left( \uu_s^h, \uu_m^h \right) \in \VV^h_s \times \VV^h_m$
		such that
		\begin{subequations}
			\begin{equation} \label{1}
			\left< -\nabla \cdot \sigma(\qq_s^h), \vv_s^h \right>_{\Omega_{s,h}} = 0
			\quad \forall \vv_s^h \in \VV^h_{s,0}
			\end{equation}
			\begin{equation} \label{2}
			\left< -\nabla \cdot \sigma(\qq_m^h), \vv_m^{h,0}\right>_{\Omega_{m,h}} = 0
			\quad \forall \vv_m^h \in \VV^h_{m,0}
			\end{equation}
			\begin{equation} \label{3}
			\left<\T \qq_s^h - \uu_m^h\left(\hat\xx_m \right), 
				\ww^h_s\right>_{\Gamma^c_{s,h}} = 0
			\quad \forall \ww^h_s \in \WW^h_s
			\end{equation}
			\begin{equation} \label{4}
			\left< -\nabla \cdot \sigma(\qq_m^h), \mathcal R_h \ww^h_m \right>_{\Omega_{m,h}}
			+\left< \sigma(\qq_m^h)\n_m - \osigma(\qq_s^h)\n_m, \ww_m^h \right>_{\Gamma^c_{m,h}}
			= 0
			\quad \forall \ww_m^h \in \WW^h_m,
			\end{equation}
		\end{subequations}
		where we have used the decomposition \eqref{H1 decomposition}, and Green's
		identity on the bilinear forms $a_i(\cdot, \cdot)$. 
		The interface condition \eqref{3} is satified due to the 
		well--known linear polynomial preserving property of the Taylor expansion.	
		It is clear that \eqref{1} and \eqref{2} are satisfied, since the second derivatives
		of all linear polynomials vanish.
		It remains to show that \eqref{4} is satisfied. Seeing that
		$-\nabla\cdot \sigma(\qq^h_m) = 0$ from the preceding argument,
		\eqref{4} becomes
		\begin{equation} \label{4a}
			\left< \sigma(\qq^h_m)\n_m - \osigma(\qq_s^h)\n_m, \ww^h_m \right>_{\Gamma^c_{m,h}}
			= 0 \quad \forall \ww^h_m \in \WW^h_m,
		\end{equation}
		The polynomial preserving property of the Zhang--Naga gradient recovery operator (cite paper
		and lemma) implies
		that the second derivatives in $\osigma(\qq^h_s)$ vanish; thus,
		\eqref{4a} must be satisfied since $\qq_s^h$ and $\qq_m^h$ are linear
		vector functions with the same analytic representation. 
	\end{proof}
	
	\begin{remark}[Second Order Accuracy]
		The second order accuracy of \eqref{compact discrete} is attributed to the
		superconvergent stress tensor constructed using the Zhang--Naga gradient
		recovery operator in conjunction with the use of the second order Taylor series expansion
		to approximate the continuous interface condition on the slave and master interfaces.
	\end{remark}
	
	\section{An Iterative Solution Method} \label{iterative solution method}
	We shall now introduce a modified Dirichlet--Neumann iterative method 
	to solve \eqref{compact discrete}.
	Let $\omega \in (0,1]$ be a relaxation parameter,
	then given an initial guess $\bg_0$, for $k=1,\ldots, \infty$,
	\begin{subequations} \label{iterative method}
		\begin{equation}
		\begin{aligned} 
		a^h_s(\uu_{s,k}^h, \vv_s^h) &= \left< \ff_s, \vv_s\right>_{\Omega_{s,h}}
		&\quad \forall \vv_s^h \in \VV^h_{s,0} \\
		\left<\uu_{s,k}^h, \ww^h_s\right>_{\Gamma^c_{s,h}} &= \left< \bg_k, \ww^h_s\right>_{\Gamma^c_{s,h}}
		&\quad \forall \ww^h_s \in \WW^h_s \\
		a^h_m(\uu_{m,k}^h, \vv_m^h) 
		- \left< \osigma(\uu_{s,k}^h)\n_m, \vv_m^h \right>_{\Gamma^c_{m,h}}
		&= \left< \ff_m, \vv_m^h\right>_{\Omega_{m,h}}
		&\quad \forall \vv_m^h \in \VV^h_m
		\end{aligned}
		\end{equation}
		\begin{equation} \label{update condition}
		\bg_{k+1} = \omega \left( \uu_{m,k}^h(\hat\xx_m)
							- \nabla \uu_{s,k}^h (\hat\xx_m - \xx_s) \right)
				+(1- \omega) \bg_k.
		\end{equation}
	\end{subequations}
	For now, we shall assume that \eqref{iterative method} is convergent. We now show
	that, under this assumption, $\lim_{k\rightarrow\infty} (\uu_{s,k}^h, \uu_{m,k}^h) 
	= (\uu^h_s, \uu^h_m)$ is the solution of \eqref{compact discrete}.
	
	For simplicity, let us denote $(\uu_s^h, \uu^h_m)\in \VV^h _s \times \VV^h_m$ as the converged
	solution of the iterative method \eqref{iterative method}. We see that in the converged limit,
	the update condition \eqref{update condition} may be represented as
	\begin{equation*} \label{update limit}
		\lim_{k\rightarrow\infty} \bg_k = \uu^h_{m}(\hat\xx_m) - 
								\nabla \uu_s^h\left( \hat\xx_m - \xx_s \right).
	\end{equation*}
	Taking $k\rightarrow \infty$ and then inserting this limit into the second equation in 
	\eqref{iterative method} yields the interface condition
	\begin{equation*}
		\left< \uu_s^h + \nabla \uu^h_s (\hat\xx_m- \xx_s) -
			\uu_m^h(\hat\xx_m), \ww_s^h  \right>_{\Gamma^c_{s,h}} = 0,
	\end{equation*}
	which is exactly the interface condition on $\Gamma^c_{s,h}$ found in
	the monolithic formulation \eqref{compact discrete}. Hence, the limit
	of this iterative solution method solves the monolithic problem under the 
	assumption that \eqref{iterative method} is convergent. 
	
	We remark that the convergence of the \eqref{iterative method} is dependent on the 
	Lam\'e parameters over the slave and master domains along with the choice of the
	relaxation parameter $\omega$. In particular, if $\mu_s$ is significantly larger than $\mu_m$,
	the iterative method becomes nonconvergent regardless of the choice of $\omega$.
	\begin{remark}
		The convergence of the iterative solution method to the monolithic
		formulation implies that there exists a nonsingular solution branch
		of \eqref{compact discrete}.
	\end{remark}
	
	\subsection{Implementation}
	In this subsection, we discuss some details pertaining to the computational
	implementation of \eqref{iterative method}. The key topics to be considered
	here are the mappings between the slave and master interfaces, treatment
	of the oblique derivative boundary condition, and the construction of the
	superconvergent gradient at the interface.
	
	\paragraph{Defining $\phi_{sm}$ and $\phi_{ms}$}
	Recall that $\phi_{sm}$ and $\phi_{ms}$ are mappings that maps $\xx_s \in \Gamma^c_{s,h}$
	to $\xx_m \in \Gamma^c_{m,h}$. In the computational setting, $\phi_{sm}$ and $\phi_{ms}$
	need only be defined over a finite number of points on each interface. One may choose to define
	these mappings for the nodes at the interface if it is desired to enforce interpolated boundary 
	conditions; or alternatively, one may choose to define these mappings on the quadrature points of the 
	interface if weak enforcement of the interface condition is desired. In this work, we choose to 
	use the nearest--neighboring node approach to define $\phi_{sm}$ and $\phi_{ms}$
	on the nodes of their respective ranges, i.e., $\Gamma^c_{m,h}$ and $\Gamma^c_{s,h}$
	respectively. The definition of this approach is defined in the
	following paragraph.
	
	Let $\zz_s$ and $\zz_m$ be arbitrary nodes of $\Gamma^h_{s,h}$ and $\Gamma^h_{m,h}$
	respectively. Then for every node $\zz_s\in \Gamma^h_{s,h}$, its nearest--neighboring
	node is defined as
	\begin{equation}
		\phi_{sm}(\zz_s) = \underset{\zz_m \in \Gamma^c_{m,h}}{\arg\min}\ d(\zz_s, \zz_m).
	\end{equation}
	Similarly, for every node $\zz_m\in \Gamma^h_{m,h}$, its nearest--neighboring node is
	defined as
	\begin{equation}
		\phi_{ms}(\zz_m) = \underset{\zz_s \in \Gamma^c_{s,h}}{\arg\min}\ d(\zz_m, \zz_s).
	\end{equation}
	If $\Gamma^c_{i,h}$, $i=s, m$, are $\mathcal O(h_i^2)$ approximations of the continuous
	interface $\Gamma^c$, then it is clear
	that the distance condition \eqref{distance condition} is satisfied using the nearest--neighboring 
	node approach. Of course, this approach 
	may yield greater than one nearest node. In such a case, it is sufficient to just chose a single node
	out of all possible node candidates that satisfies the minimum distance criterion.
	\begin{remark}
		In the literature, there are methods that allows one to get a closer neighboring point
		on the neighboring interface. Some methods are detailed in \cite{de2008comparison}.
	\end{remark}

	\paragraph{The Dirichlet Subproblem}	
	In \eqref{iterative method}, the Dirichlet interface condition on $\Gamma^c_{h,s}$ is enforced weakly. 
	While this provides desirable mathematical properties (i.e., the Dirichlet data may be less regular), 
	in practice, one may choose to instead enforce this condition in the strong form,
	i.e.,
	\begin{equation}
		\uu^h_{s,k} \big|_{\Gamma^c_{s,h}} = \bg_k.
	\end{equation}
	In our numerical experiments, this is the condition that we enforce. Of course, we must now
	introduce some sort of operator that maps the 
	$\uu^h_{m,k} - \nabla \uu^h_{s,k}(\hat\xx_m-\xx_s)$
	term in the update condition of \eqref{iterative method} into $\WW^h_s$, which is again 
	the piecewise linear approximation space over $\Gamma^c_{s,h}$. In this work, we choose
	to interpolate the values of this term over the nodes of the $\Gamma^c_{s.h}$. Subsequently,
	the update condition becomes
	\begin{equation}
		\bg_{k+1} = \omega \mathcal I_{s,h} \left( \uu_{m,k}^h(\hat\xx_m)
							- \nabla \uu_{s,k}^h \cdot (\hat\xx_m - \xx_s) \right)
				+(1- \omega) \bg_k \quad \forall \ww^h_s \in \WW^h_s,
	\end{equation}
	where $\mathcal I_{s,h}$ is the piecewise linear interpolation operator over $\Gamma^c_{s,h}$.
	This interpolation is implied in the standard
	method of enforcing Dirichlet conditions in the finite element method.
	
	\paragraph{The Neumann Subproblem}
	At first glance, it may appear that we must define the superconvergent	
	Jacobian, $\GG\uu^h_{s,k}$, everywhere in $\Omega_{s,h}$; however, in this application, 
	$\GG\uu^h_{s,k}$ is used only in the construction of $\osigma(\uu^h_{s,k})$. 
	Therefore, we only require
	that $\GG\uu^h_{s,k}$  is defined only over the triangles that contain the nodes
	in $\Gamma^c_{s,h}$. We shall call this region $\overline {\mathcal T}_{s,h}$. 
	An illustration of $\overline{\mathcal T}_{s,h}$ is presented in Fig. \ref{recovery layer}.
	
	After constructing $\GG\uu^h_{s,k}$ over $\overline{\mathcal T}_{s,h}$, we 
	we may construct $\overline \varepsilon(\uu^h_{s,k})$, and subsequently $\osigma(\uu^h_{s,k})$.
	While there is no issue approximating 
	$\left< \osigma(\uu^h_{s,k})\n_m, \vv_m^h \right>_{\Gamma^c_{m,h}}$ by using a quadrature rule,
	in this work, we choose to interpolate the values of $\osigma(\uu^h_{s,k})$ at the 
	nodes of $\Gamma^c_{m,h}$ and then use a quadrature rule to compute the 
	natural boundary condition term of the variational form.
	
	\begin{figure}[htbp]
		\centering
		\resizebox{0.1\textheight}{!}{
		\input{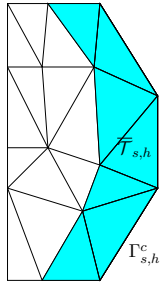}
		}
		\caption{An illustration of the region $\overline{\mathcal T}_{s,h}$.}
		\label{recovery layer}
	\end{figure}
	
	\paragraph{A Complete Description of the Algorithm}
	In the preceding paragraphs, we have given the mathematical statement of the
	modified Dirichlet--Neumann iterative approach and provided some key details
	in the implementation of this algorithm. We shall now
	provide a sketch of what a code implementation of \eqref{iterative method}
	would look like.
	\begin{algorithm}[Modified Dirichlet--Neumann]
	Define $M = \dim(\VV^h_{s,0})$ and $M' = \dim(\VV^h_m)$, and $\psi_j, j=1, \ldots, M$ and
	$\Psi_{j'}, j'=1,\ldots, M'$ as the bases of $\VV^h_{s,0}$ and $\VV^h_m$ respectively.
	Given triangulations $\Omega_{s,h}$ and $\Omega_{m,h}$, an error tolerance 
	$\delta\in \mathbb R^+$,
	and a relaxation parameter $\omega \in  (0,1]$,
	\begin{enumerate}
		\item Determine $\overline{\mathcal T}_{s,h}$.
		\item Obtain the element patch information for the nodes in 
			$\overline{\mathcal T}_{s,h}$ for the 
			superconvergent gradient algorithm.
		\item Get the nearest--neighboring nodes of the interfaces 
				$\Gamma^c_{s,h}$ and $\Gamma^c_{m,h}$.
		\item Set initial guess $\bg_0$.
		\item While $\| \bg_{k+1} - \bg_k \|_{L^2(\Gamma^c_{s,h})} > \delta$:
		\begin{enumerate}
			\item Solve
			\begin{equation*}
				a^h_s(\uu_{s,k}^h, \psi_j) = \left< \ff_s, \psi_j\right>_{\Omega_{s,h}}
				\quad \textrm{for } j= 1,\ldots, M
			\end{equation*}
			with interface condition
			\begin{equation*}
				\uu_{s,k}^h \big|_{\Gamma^c_{s,h}}= \bg_k.
			\end{equation*}
			\item Compute $\tensor{J} = \GG \uu^h_{s,k}$ over $\overline{\mathcal T}_{s,h}$.
			\item On $\Gamma^c_{m,h}$:
			\begin{enumerate}
			\item Compute $\TT \tensor{J}$ from formula \eqref{taylor jacobian}.
			\item Compute $\overline\varepsilon(\uu^h_{s,k}) = \frac12 \left[
						\TT \tensor{J} + \left(\TT \tensor{J}\right)^T \right]$.
			\item Compute $\osigma(\uu^h_{s,k}) = \lambda_s \textrm{ tr }\overline\varepsilon
							(\uu^h_{s,k}) + 2\mu_s \overline\varepsilon(\uu^h_{s,k})$.
			\end{enumerate}
			\item Solve
			\begin{equation*}
					a^h_m(\uu_{m,k}^h, \Psi_{j'}) 
		- \left<\mathcal I_{m,h} \osigma(\uu_{s,k}^h)\n_m, \Psi_{j'} \right>_{\Gamma^c_{m,h}}
		= \left< \ff_m, \Psi_{j'}\right>_{\Omega_{m,h}}
		\quad \textrm{for } j' = 1,\ldots, M'.
			\end{equation*}
			\item Compute
				\begin{equation*}
						\bg_{k+1} = \omega \mathcal I_{s,h}\left( \uu_{m,k}^h(\hat\xx_m)
							- \nabla \uu_{s,k}^h \cdot (\hat\xx_m- \xx_s) \right)
				+(1- \omega) \bg_k.
				\end{equation*}
		\end{enumerate}
	\end{enumerate}
	\end{algorithm}

	\section{Numerical Results} \label{Numerical Results}
	In this section, we will demonstrate that the iterative coupling procedure 
	presented in the previous section allows us to achieve the error bounds
	expected from the piecewise linear finite element method.
	
	For simplicity, let us define $\uu = \left( \uu_s, \uu_m \right)$
	as the solution of the continuous interface coupling problem \eqref{galerkin form},
	and 	$\uu^h = \left( \uu_s^h, \uu_m^h \right)$ as the converged solution of the
	iterative formulation \eqref{iterative method}. In addition, let 
	$\Omega_h = \Omega_{s,h} \cup \Omega_{m,h}$.
	We will report the error in the broken norms defined as
	\begin{equation*}
		\| \vv \|_{L^2(\Omega_h)} = 
			\left(\| \vv_s \|_{L^2(\Omega_{s,h})}^2 + \| \vv_m\|_{L^2(\Omega_{m,h})}^2\right)^\frac12
	\end{equation*}
	and
	\begin{equation*}
		\| \vv \|_{H^1(\Omega_h)} = 
			\left(\| \vv_s \|_{H^1(\Omega_{s,h})}^2 + \| \vv_m \|_{H^1(\Omega_{m,h})}^2\right)^\frac12
	\end{equation*}
	for all functions $\vv = (\vv_s, \vv_m)$ in $H^1(\Omega_{s,h}) \times H^1(\Omega_{m,h})$.
	In addition, we will use the notation $\#(\Gamma^c_{h,i})$, $i=s, m$ to denote the number
	of elements on $\Gamma^c_{h,i}$. Also, 
	again, we define $h:=\max\{h_s, h_m\}$ as the maximum element diameter in the 
	of both triangulations $\Omega_{s,h}$ and $\Omega_{m,h}$.
 	Finally, let $K$ be the number of iterations required to reach convergence. 	
	
	The convergence results in the following numerical experiments will be structured
	in the following manner. First, we shall discuss the domain decomposition problem,
	i.e., where the Lam\'e parameters of the subproblems are equal. In this test problem,
	we will first demonstrate that the linear patch test is passed for subdomains with 
	noncoincident curved interfaces, and then we will demonstrate that we get the expected
	finite element convergence rates for piecewise linear approximation spaces. Second,
	we will present the convergence rates for the general interface coupling problem
	where the Lam\'e parameters of the solids are not necessarily equal. The error tolerance
	$\delta = 1\sci{-6}$ was used in all of the following numerical experiments. 
	For convenience, we call the ratio between the number of elements on the slave
	interface and the number of elements on the master interface the ``slave : master ratio''.
	
	\begin{figure}
		\centering
		\includegraphics[width=0.45\textwidth]{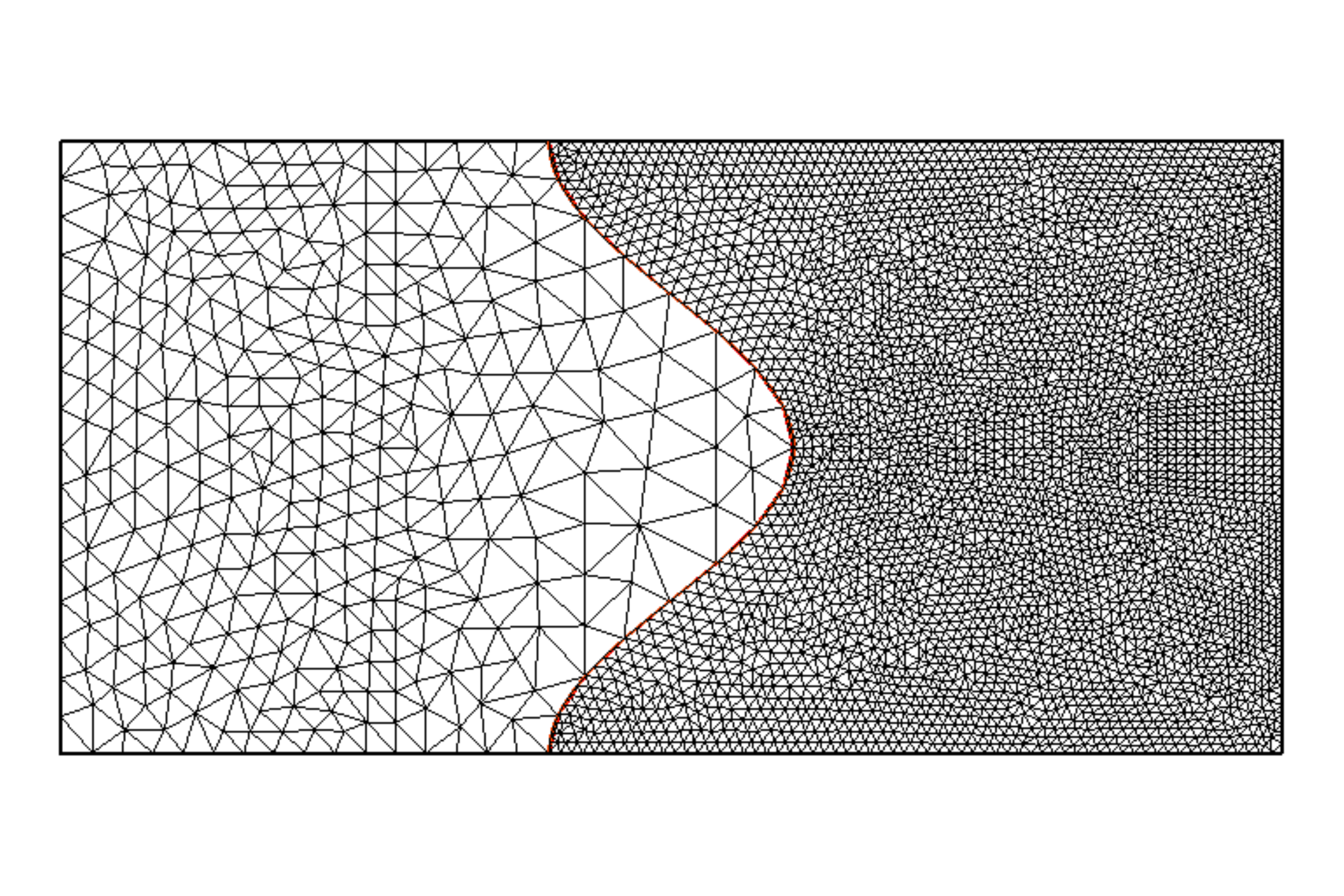}
		\caption{An illustration of a spatially noncoincident mesh used in 
		the numerical examples.}
	\end{figure}
	
	\subsection{The Domain Decomposition Problem}
	For the sake of convenience, we set the Lam\'e parameters $(\lambda, \mu) = (1, 1)$
	on both subdomains. The coupled problem domain in this test problem is a 
	rectangle that occupies the region $\Omega_h = \left[-1, 1\right] \times [0, 1]$ with
	discrete interfaces generated from the analytic representation $-\frac15 \cos(2\pi y)$,
	where $y\in[0, 1]$ at $x = 0$. $\Omega_{s,h}$ is taken to be the triangulation
	of the subdomain on the left side of the continuous interface, and $\Omega_{m,h}$ is
	taken to be the triangulation of the subdomain on the right side of the continuous interface.
	
	\subsubsection{The Linear Patch Test}
	In this numerical experiment, we determine whether \eqref{iterative method} can
	recover the linear vector function $\uu = \left[x+y, x+y\right]^T$ as an exact solution
	for the homogeneous problem. The results for some representative subdomain configurations
	are presented in Table \ref{patch test table}. It is demonstrated that 
	\eqref{iterative method} for both spatially coincident and noncoincident interfaces. 
	The relaxation parameter in the algorithm was set to $\omega = 0.7$.
	
	\begin{table}[htbp]
        \begin{center}
        \begin{tabular}{|c|c|c|c|}
        		\hline
        		$\#(\Gamma^c_{h,s})$ & $\#(\Gamma^c_{h,m})$ & $\|\uu - \uu^h\|_{L^2(\Omega_h)}$ & $K$\\
		\hline
		$25$ & $25$ & $2.4960\sci{-7}$ & $28$\\
		$100$ & $100$ & $2.2622\sci{-7}$ & $26$ \\
		\hline
		$25$ & $50$ & $3.0596\sci{-7}$ & $28$ \\
		$25$ & $100$ & $3.1428\sci{-7}$ & $29$ \\
		\hline
		$50$ & $25$ & $2.7181\sci{-7}$ & $27$ \\
		$100$ & $25$ & $4.8236\sci{-7}$ & $36$\\
		\hline
        \end{tabular}
        \caption{Linear Patch Test. The iterative method \eqref{iterative method} recovers the linear patch
        test.}
        \end{center}
               \label{patch test table}
        \end{table}
        
        \begin{figure}
        		\centering
        		\includegraphics[width=0.45\textwidth]{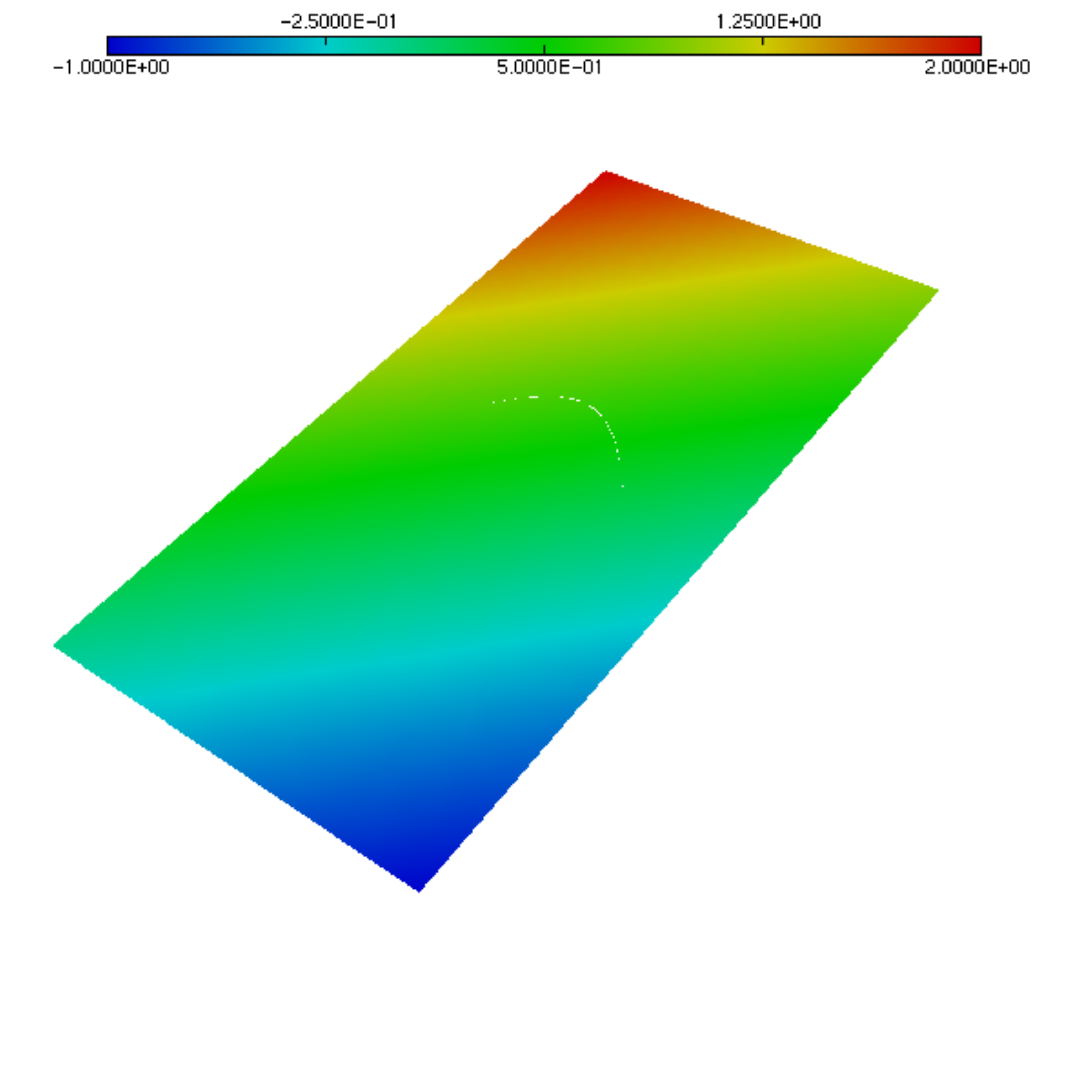}
		\includegraphics[width=0.45\textwidth]{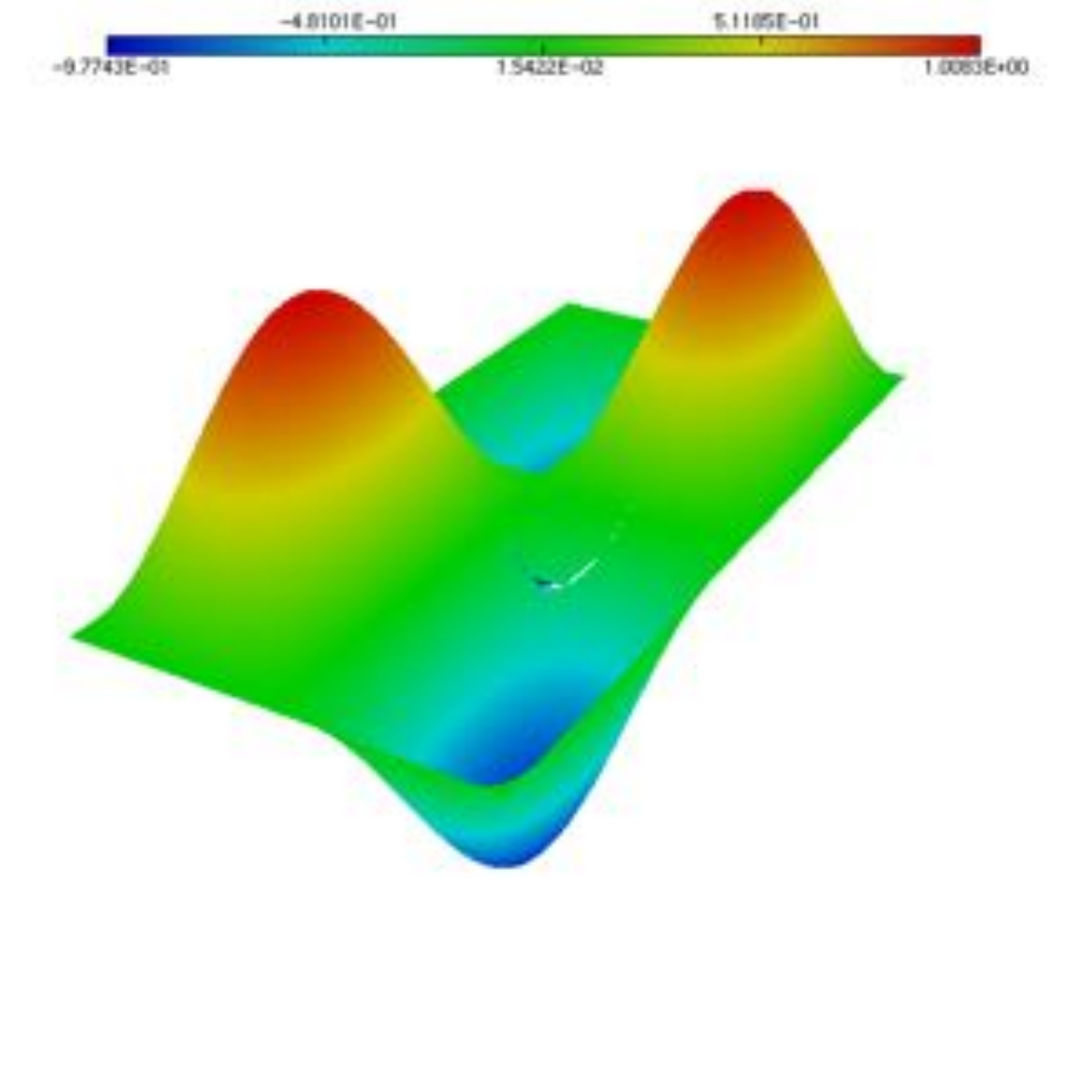}
		\caption{Left: The x--displacement solution of the linear patch test.
		Right: $ \sin(\pi x)\sin(2\pi y)$ as computed from the x--displacement
		of the solution of the iterative method.}
        \end{figure}
        
	\subsubsection{Convergence to a Manufactured Solution}
	Here, we determine the rate of convergence of the solution of the discrete
	problem to the manufactured solution $\uu = \left[ \sin(\pi x)\sin(2\pi y), \sin(\pi x)\sin(2\pi y) \right]^T$.
	First, we present the convergence results for the case where $\Gamma^c_{h,i}$, $i=s,m$,
	coincide, and then we shall present the convergence history for the
	spatially noncoincident cases, where the ratio between the number of elements
	on the master interface and the slave interface vary. 
	We have demonstrated in the Tables 2--4 that the iterative method presented in the
	previous section converges optimally for piecewise linear Lagrange elements. An illustration
	of the convergence rate can be found in Fig. \ref{DD Convergence}.
	
	\begin{figure}[htbp]
        \begin{center}
		\includegraphics{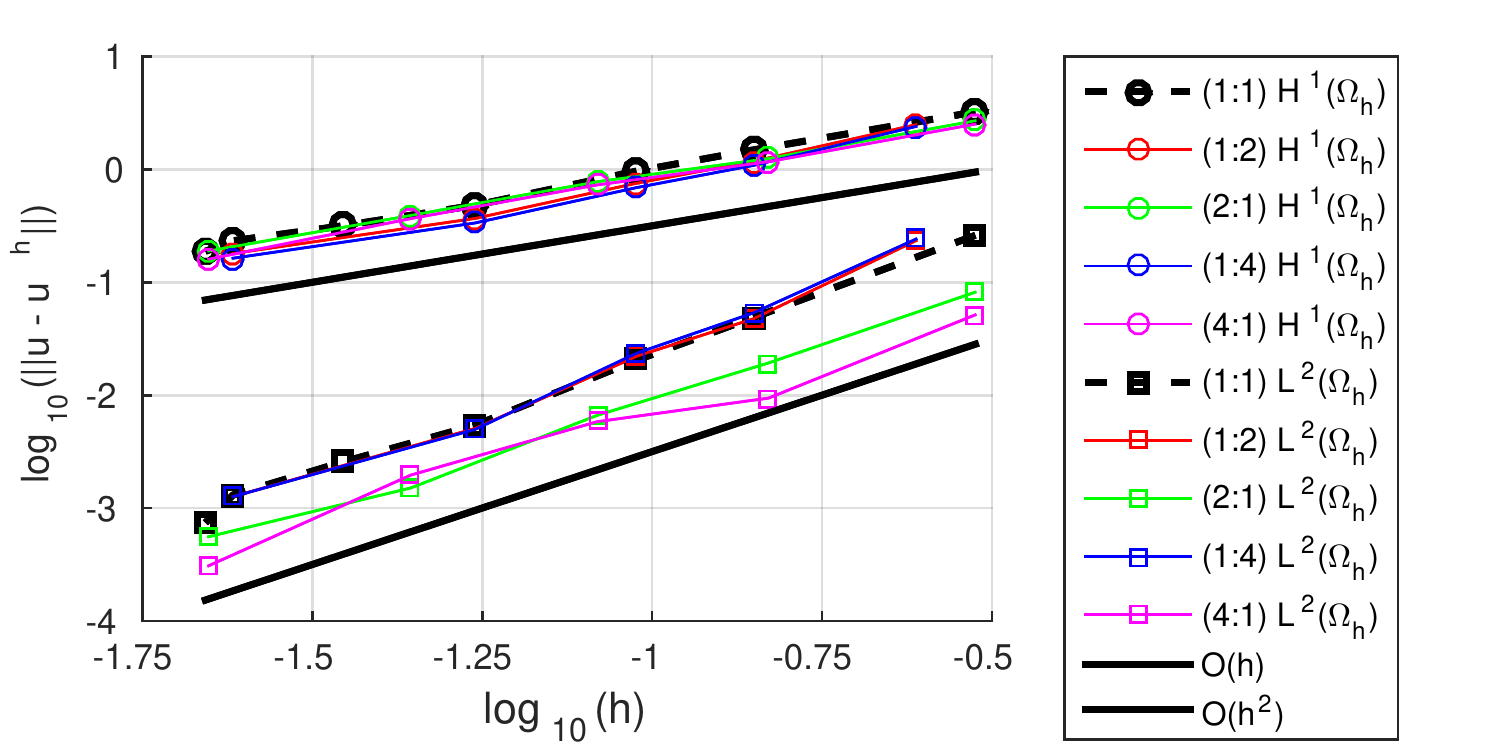}        
        \caption{ $H^1(\Omega_h)$ and $L^2(\Omega_h)$ convergence history illustration for
        		     the domain decomposition problem.}
        \label{DD Convergence}
        \end{center}
        \end{figure}

	\begin{table}[htbp]
        \begin{center}
        \begin{tabular}{|c|c|c|c|c|c|}
        		\hline
        		$\#(\Gamma^c_{h,s})$ & $\#(\Gamma^c_{h,m})$ & $h$ & $\|\uu - \uu^h\|_{L^2(\Omega_h)}$ &
		$\| \uu - \uu^h \|_{H^1(\Omega_h)}$ & $K$\\
		\hline
		$8$ & $8$ & $0.244511$ & $0.260548$ & $3.22675$ & $26$\\
		$16$ & $16$ & $0.140988$ & $0.0487357$ & $1.48733$ & $26$ \\
		$25$ & $25$ & $0.0948402$ & $0.0203974$ & $0.964635$ & $22$ \\
		$50$ & $50$ & $0.0548697$ & $0.00540889$ & $0.480094$ & $17$\\
		$75$ & $75$ & $0.0350554$ & $0.00255794$ & $0.316482$ & $17$ \\
		$100$ & $100$  & $0.0241639$ & $0.00130554$ & $0.233027$ & $17$ \\
		$125$ & $125$ & $0.0220104$ & $0.000741496$ & $0.188069$ & $18$\\
		\hline
		\multicolumn{3}{|r|}{} 
		& $L^2(\Omega_h)$ Rate &
		\multicolumn{2}{|c|}{2.17} \\
		\hline
		\multicolumn{3}{|c|}{}
		&$H^1(\Omega_h)$ Rate & 
		\multicolumn{2}{|c|}{1.01} \\
		\hline
        \end{tabular}
        \caption{Domain Decomposition Problem: Spatially Coincident Interface. }
        \end{center}
        \label{DD table 1}
        \end{table}

	\begin{table}[htbp]
        \begin{center}
        \begin{tabular}{|c|c|c|c|c|c|}
        		\hline
        		$\#(\Gamma^c_{h,s})$ & $\#(\Gamma^c_{h,m})$ & $h$ & $\|\uu - \uu^h\|_{L^2(\Omega_h)}$ &
		$\| \uu - \uu^h \|_{H^1(\Omega_h)}$ & $K$\\
		\hline
		$8$ & $16$ & $0.244511$ & $0.233747$ & $2.48788$ & $26$\\
		$16$ & $32$ & $0.140988$ & $0.0475133$ & $1.17863$ & $25$ \\
		$25$ & $50$ & $0.0948402$ & $0.0220765$ & $0.753572$ & $22$ \\
		$50$ & $100$ & $0.0548697$ & $0.00507153$ & $0.367443$ & $18$\\
		$100$ & $200$ & $0.0241639$ & $0.00126299$ & $0.179719$ & $18$\\
		\hline
		\multicolumn{3}{|r|}{} 
		& $L^2(\Omega_h)$ Rate &
		\multicolumn{2}{|c|}{2.25} \\
		\hline
		\multicolumn{3}{|c|}{}
		&$H^1(\Omega_h)$ Rate & 
		\multicolumn{2}{|c|}{1.14} \\
		\hline
        \end{tabular}
        \begin{tabular}{|c|c|c|c|c|c|}
        		\hline
        		$\#(\Gamma^c_{h,s})$ & $\#(\Gamma^c_{h,m})$ & $h$ & $\|\uu - \uu^h\|_{L^2(\Omega_h)}$ &
		$\| \uu - \uu^h \|_{H^1(\Omega_h)}$ & $K$\\
		\hline
		$16$ & $8$ & $0.29873$ & $0.081574$ & $2.69409$ & $33$ \\
		$32$ & $16$ & $0.148393$ & $0.0192996$ & $1.24801$ & $26$ \\
		$50$ & $25$ & $0.0830141$ & $0.00662555$ & $0.776856$ & $23$ \\
		$100$ & $50$ & $0.0441468$ & $0.00150745$ & $0.39077$ & $24$ \\
		$200$ & $100$ & $0.0222507$ & $0.000557604$ & $0.189376$ & $22$ \\
		\hline
		\multicolumn{3}{|r|}{} 
		& $L^2(\Omega_h)$ Rate &
		\multicolumn{2}{|c|}{1.95} \\
		\hline
		\multicolumn{3}{|c|}{}
		&$H^1(\Omega_h)$ Rate & 
		\multicolumn{2}{|c|}{1.01} \\
		\hline
        \end{tabular}
        \caption{Domain Decomposition Problem. Top: $(1:2)$ slave : master ratio. Bottom:
        			$(2:1)$ slave : master ratio.}
        \end{center}
        \label{DD table 2}
        \end{table}

	\begin{table}[htbp]
        \begin{center}
        \begin{tabular}{|c|c|c|c|c|c|}
        		\hline
        		$\#(\Gamma^c_{h,s})$ & $\#(\Gamma^c_{h,m})$ & $h$ & $\|\uu - \uu^h\|_{L^2(\Omega_h)}$ &
		$\| \uu - \uu^h \|_{H^1(\Omega_h)}$ & $K$\\
		\hline
		$8$ & $32$ & $0.244511$ & $0.241335$ & $2.39941$ & $25$\\
		$16$ & $64$ & $0.140988$ & $0.0535691$ & $1.08416$ & $24$ \\
		$25$ & $100$ & $0.0948402$ & $0.0235007$ & $0.689027$ & $23$ \\
		$50$ & $200$ & $0.0548697$ & $0.00497135$ & $0.334919$ & $19$ \\
		$100$ & $400$ & $0.0241639$ & $0.00126203$ & $0.163256$ & $18$ \\
		\hline
		\multicolumn{3}{|r|}{} 
		& $L^2(\Omega_h)$ Rate &
		\multicolumn{2}{|c|}{2.29} \\
		\hline
		\multicolumn{3}{|c|}{}
		&$H^1(\Omega_h)$ Rate & 
		\multicolumn{2}{|c|}{1.16} \\
		\hline
        \end{tabular}
        \begin{tabular}{|c|c|c|c|c|c|}
        		\hline
        		$\#(\Gamma^c_{h,s})$ & $\#(\Gamma^c_{h,m})$ & $h$ & $\|\uu - \uu^h\|_{L^2(\Omega_h)}$ &
		$\| \uu - \uu^h \|_{H^1(\Omega_h)}$ & $K$\\
		\hline
		$32$ & $8$ & $0.29873$ & $0.0513822$ & $2.51413$ & $36$ \\
		$64$ & $16$ & $0.148393$ & $0.0094199$ & $1.1693$ & $31$ \\
		$100$ & $25$ & $0.0830141$ & $0.00582067$ & $0.728018$ & $36$\\
		$200$ & $50$ & $0.0441468$ & $0.0019599$ & $0.366355$ & $52$ \\
		$400$ & $100$ & $0.0222507$ & $0.00030831$ &  $0.160208$ & $32$ \\
		\hline
		\multicolumn{3}{|r|}{} 
		& $L^2(\Omega_h)$ Rate &
		\multicolumn{2}{|c|}{1.85} \\
		\hline
		\multicolumn{3}{|c|}{}
		&$H^1(\Omega_h)$ Rate & 
		\multicolumn{2}{|c|}{0.98} \\
		\hline
        \end{tabular}
        \caption{Domain Decomposition Problem. Top: $(1:4)$ slave : master ratio. 
        			Bottom: $(4:1)$ slave : master ratio.}
        \end{center}
        \label{DD table 3}
        \end{table}

	\subsection{Discontinuous Lam\'e Parameters}
	In this test problem, we consider the coupling problem \eqref{compact discrete}
	with Lam\'e parameters $(\lambda_s, \mu_s)=(2,1)$ and $(\lambda_m, \mu_m) = (2,1)$.
	The coupled problem domain in this test case is the same as the domain considered in the
	previous subsection, i.e., a rectangle that occupies the region $\Omega_h = [-1, 1]\times [0,1]$
	with discrete interfaces generated from the analytic representation $-\frac15 \cos(2\pi y)$,
	where $y\in[0,1]$ at $x=0$. $\Omega_{s,h}$ is again taken to be the triangulation of the
	subdomain at the left side of the continuous interface, and $\Omega_m,h$ is taken to 
	be the triangulation of the subdomain on the right side of the continuous interface.
	We have set $\ff_s = \left[10, 10\right]^T$ and $\ff_m = \left[-10, -10 \right]^T$ and 
	$\uu_s = 0$ on $\Gamma_{s,h}$ and $\uu_m = 0 $ on $\Gamma_{m,h}$. 
	The relaxation parameter used in this numerical experiment was chosen to be $\omega = 0.7$.
	We test the rate of convergence of our iterative method against an 
	``overkilled'' solution, i.e., a discrete solution computed on a mesh size of $5.499\sci{-3}$.
	This solution was computed monolithically using standard piecewise linear Lagrange elements.

	From the numerical results gathered from this experiment, it is clear that the 
	iterative method presented in the previous section is a general method that
	is second order accurate in $L^2(\Omega_h)$ and first order accurate in 
	$H^1(\Omega_h)$.
	
	\begin{figure}[htbp]
		\centering
		\includegraphics[width=0.45\textwidth]{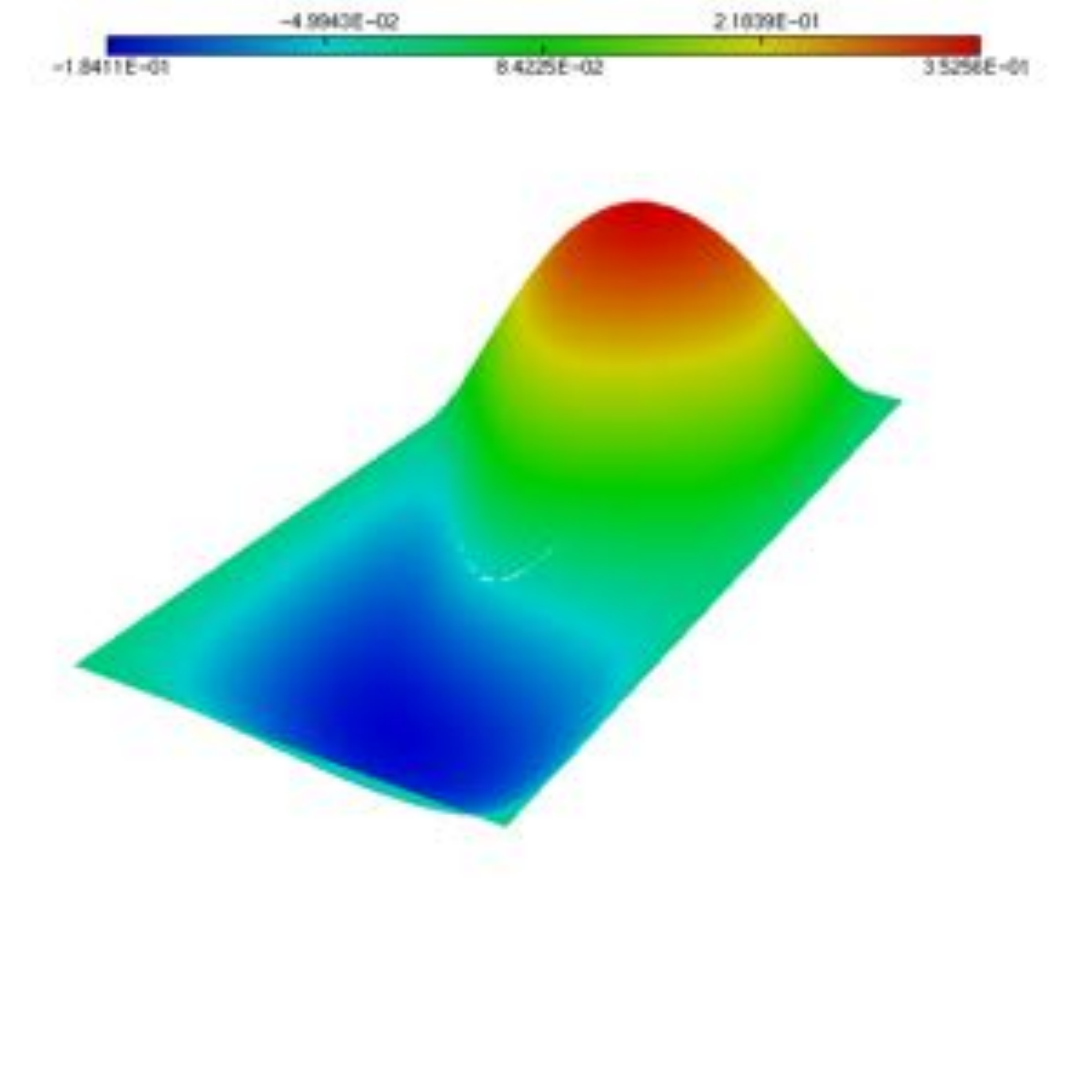}
		\includegraphics[width=0.45\textwidth]{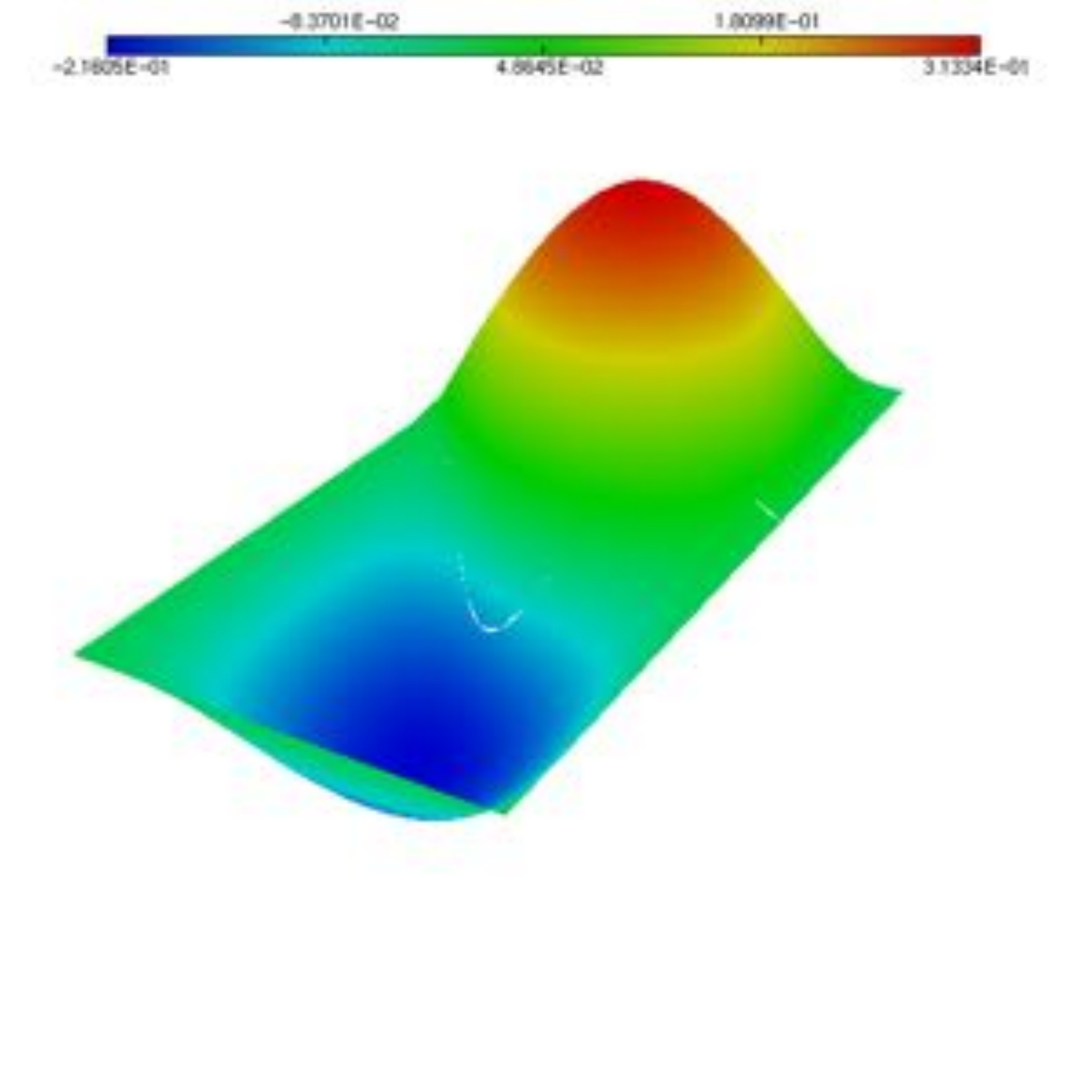}
		\caption{The converged solution of the iterative method for the discontinuous
		Lam\'e parameter test problem.  Left: x--displacement. Right: y--displacement.}
	\end{figure}
	\begin{figure}[htbp]
		\centering
		\includegraphics{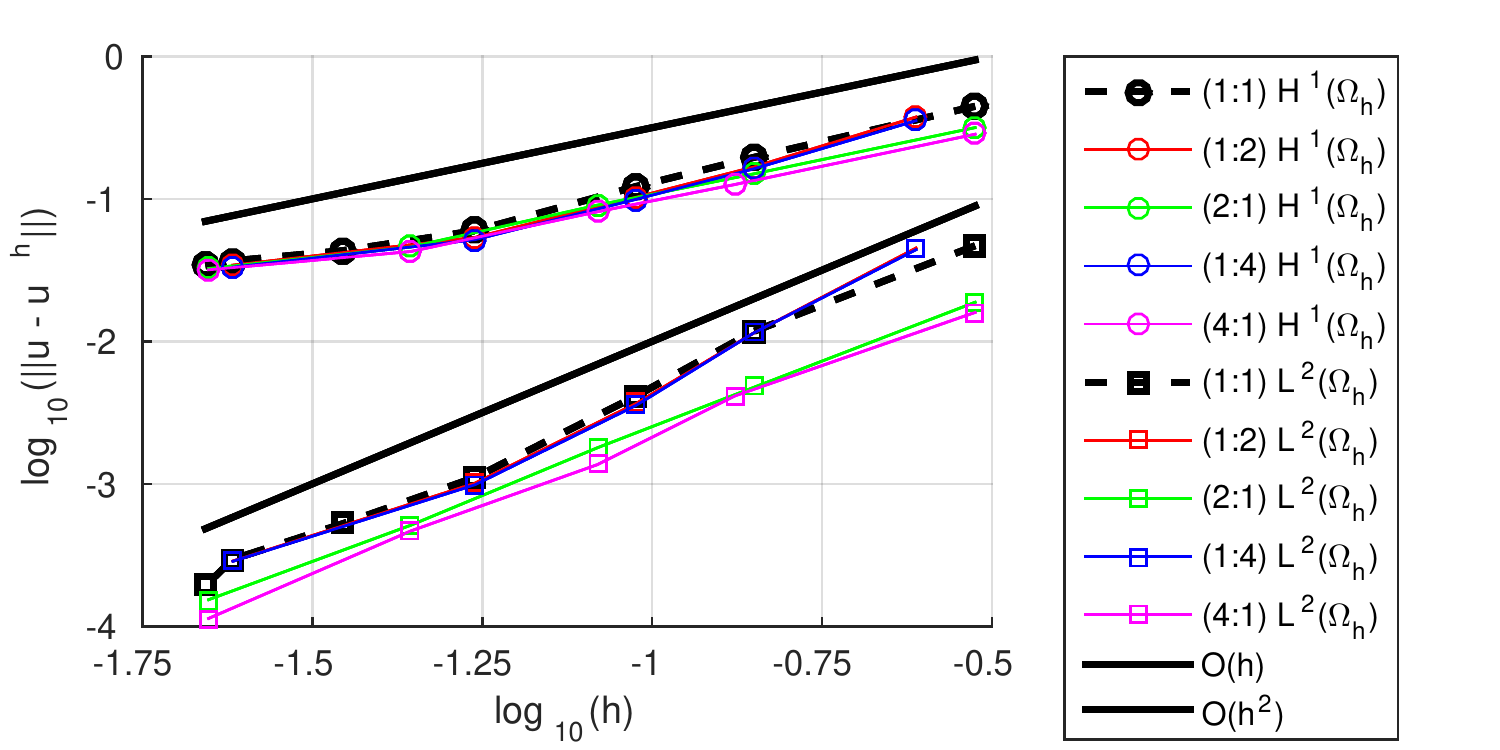}
		\caption{$H^1(\Omega_h)$ and $L^2(\Omega_h)$ convergence for the discontinuous
				Lam\'e parameter test problem.}
	\end{figure}
	\begin{table}[htbp]
        \begin{center}
        \begin{tabular}{|c|c|c|c|c|c|}
        		\hline
        		$\#(\Gamma^c_{h,s})$ & $\#(\Gamma^c_{h,m})$ & $h$ & $\|\uu - \uu^h\|_{L^2(\Omega_h)}$ &
		$\| \uu - \uu^h \|_{H^1(\Omega_h)}$ & $K$\\
		\hline
		$8$ & $8$ & $0.29873$ & $0.0469285$ & $0.446852$ & $14$ \\
		$16$ & $16$ & $0.140988$ & $0.0117446$ & $0.195434$  & $13$\\
		$25$ & $25$ & $0.0948402$ & $0.00415211$ & $0.122464$ & $12$\\
		$50$ & $50$ & $0.0548697$ & $0.00110825$ & $0.0597868$ & $14$\\
		$75$ & $75$ & $0.0350554$ & $0.000525341$ & $0.0433184$ & $15$\\
		$100$ & $100$ & $0.0241639$ & $0.000295744$ & $0.0367863$ & $15$ \\
		$125$ & $125$ & $0.0220104$ & $0.000192478$ & $0.0338451$ & $16$\\
		\hline
		\multicolumn{3}{|r|}{} 
		& $L^2(\Omega_h)$ Rate &
		\multicolumn{2}{|c|}{ 2.10} \\
		\hline
		\multicolumn{3}{|c|}{}
		&$H^1(\Omega_h)$ Rate & 
		\multicolumn{2}{|c|}{1.00} \\
		\hline
        \end{tabular}
        \caption{Different Lam\'e Parameters: Spatially Coincident Interface. }
        \end{center}
        \label{Different Lame Table 1}
        \end{table}

	\begin{table}[htbp]
        \begin{center}
        \begin{tabular}{|c|c|c|c|c|c|}
        		\hline
        		$\#(\Gamma^c_{h,s})$ & $\#(\Gamma^c_{h,m})$ & $h$ & $\|\uu - \uu^h\|_{L^2(\Omega_h)}$ &
		$\| \uu - \uu^h \|_{H^1(\Omega_h)}$ & $K$\\
		\hline
		$8$ & $16$ & $0.244511$ & $0.0451623$ & $0.375507$ & $14$ \\
		$16$ & $32$ & $0.140988$ & $0.0114271$ & $0.167331$ & $13$ \\
		$25$ & $50$ & $0.0948402$ & $0.0036823$ & $0.103027$ & $13$ \\
		$50$ & $100$ & $0.0548697$ & $0.00101088$ & $0.0536388$ & $15$ \\
		$100$ & $200$ & $0.0241639$ & $0.00028725$ & $0.0342972$ & $16$ \\
		\hline
		\multicolumn{3}{|r|}{} 
		& $L^2(\Omega_h)$ Rate &
		\multicolumn{2}{|c|}{2.01} \\
		\hline
		\multicolumn{3}{|c|}{}
		&$H^1(\Omega_h)$ Rate & 
		\multicolumn{2}{|c|}{1.04} \\
		\hline
        \end{tabular}
        \begin{tabular}{|c|c|c|c|c|c|}
        		\hline
        		$\#(\Gamma^c_{h,s})$ & $\#(\Gamma^c_{h,m})$ & $h$ & $\|\uu - \uu^h\|_{L^2(\Omega_h)}$ &
		$\| \uu - \uu^h \|_{H^1(\Omega_h)}$ & $K$\\
		\hline
		$16$ & $8$ & $0.29873$ & $0.018827$ & $0.316666$ & $19$ \\
		$32$ & $16$ & $0.148393$ & $0.00480391$ & $0.150653$ & $15$ \\
		$50$ & $25$ & $0.0830141$ & $0.00178996$ & $0.0903401$ & $14$ \\
		$100$ & $50$ & $0.0441468$ & $0.000512853$ & $0.0460626$ & $16$ \\
		$200$ & $100$ & $0.0222507$ & $0.000153844$ & $0.0325039$ & $19$ \\
		\hline
		\multicolumn{3}{|r|}{} 
		& $L^2(\Omega_h)$ Rate &
		\multicolumn{2}{|c|}{1.86} \\
		\hline
		\multicolumn{3}{|c|}{}
		&$H^1(\Omega_h)$ Rate & 
		\multicolumn{2}{|c|}{0.90} \\
		\hline
        \end{tabular}
        \caption{Different Lam\'e Parameters. Top: $(1:2)$ slave : master ratio. 
        			 Bottom: $(2:1)$ slave : master ratio.}
        \end{center}
        \label{Different Lame Table 2}
        \end{table}

	\begin{table}[htbp]
        \begin{center}
        \begin{tabular}{|c|c|c|c|c|c|}
        		\hline
        		$\#(\Gamma^c_{h,s})$ & $\#(\Gamma^c_{h,m})$ & $h$ & $\|\uu - \uu^h\|_{L^2(\Omega_h)}$ &
		$\| \uu - \uu^h \|_{H^1(\Omega_h)}$ & $K$\\
		\hline
		$8$ & $32$ & $0.244511$ & $0.044073$ & $0.354096$ & $14$ \\
		$16$ & $64$ & $0.140988$ & $0.0114338$ & $0.162604$ & $13$ \\
		$25$ & $100$ & $0.0948402$ & $0.00356991$ & $0.0996944$ & $13$\\
		$50$ & $200$ & $0.0548697$ & $0.000985668$ & $0.0517424$ & $15$ \\
		$100$ & $400$ & $0.0241639$ & $0.000286343$ & $0.0332221$ & $16$ \\
		\hline
		\multicolumn{3}{|r|}{} 
		& $L^2(\Omega_h)$ Rate &
		\multicolumn{2}{|c|}{2.21} \\
		\hline
		\multicolumn{3}{|c|}{}
		&$H^1(\Omega_h)$ Rate & 
		\multicolumn{2}{|c|}{1.04} \\
		\hline
        \end{tabular}
        \begin{tabular}{|c|c|c|c|c|c|}
        		\hline
        		$\#(\Gamma^c_{h,s})$ & $\#(\Gamma^c_{h,m})$ & $h$ & $\|\uu - \uu^h\|_{L^2(\Omega_h)}$ &
		$\| \uu - \uu^h \|_{H^1(\Omega_h)}$ & $K$\\
		\hline
		$32$ & $8$ & $0.29873$ & $0.0159833$ & $0.284411$ & $34$ \\
		$64$ & $16$ & $0.132189$ & $0.00413673$ & $0.126088$ & $40$ \\
		$100$ & $25$ & $0.0830141$ & $0.0013636$ & $0.0812393$ & $27$ \\
		$200$ & $50$ & $0.0441468$ & $0.000467162$ & $0.0427591$ & $29$ \\
		$400$ & $100$ & $0.0222507$ & $0.00011361$ & $0.0317956$ & $34$ \\ 
		\hline
		\multicolumn{3}{|r|}{}  
		& $L^2(\Omega_h)$ Rate &
		\multicolumn{2}{|c|}{1.91} \\
		\hline
		\multicolumn{3}{|c|}{}
		&$H^1(\Omega_h)$ Rate & 
		\multicolumn{2}{|c|}{0.87} \\
		\hline
        \end{tabular}
        \caption{Different Lam\'e Parameters. Top: $(1:4)$ slave : master ratio. Bottom: 
        								$(4:1)$ slave : master ratio.}
        \end{center}
        \label{Different Lame Table 3}
        \end{table}
	
	\section{Conclusion} \label{conclusion}
	In this paper, we have described a coupling formulation for coupling the equations of linear
	elasticity in the discrete setting where the discrete interface is noncoincident. In addition, we have
	proposed an iterative solution method that solves this coupling formulation. In the numerical 
	experiments presented in the numerical results section that the converged solution to our 
	iterative solution method satisfies the expected error bounds for piecewise linear finite elements.
	The numerical results presented in the previous section implies that this method is applicable
	for both domain decomposition problems and interface coupling problems where the 
	Lam\'e parameters differ on the problem subdomains. 
	

	In the monolithic formulation \eqref{compact discrete}, we have approximated the
	Dirichlet subproblem with an oblique derivative problem. We anticipate that this
	approximation may also be used to improve the error bounds of finite element
	approximations on unfitted meshes, as was done in \cite{cockburnextension}.
	In addition, we anticipate that, in the future, we will generalize this method to
	$k$--order polynomial approximation spaces and to time--dependent problems.
	A mathematical treatment on this method applied to the Poisson equation
	will be the topic of an upcoming work.
	
	\section{Appendix: Definition of the Zhang--Naga Gradient Recovery Operator}
Let us first define the notion of an element patch.
Let $z_0 \in \mathcal N$ be a node in $\overline{\mathcal T}_{s,h}$ and $T\in \overline{\mathcal T}_{s,h}$ be a triangle in $\overline{\mathcal T}_{s,h}$, then we can 
define an initial element patch as the following
\begin{equation}
	\mathcal K^0_{z_0} = \left\{ T\in \overline{\mathcal T}_{s,h}; z_0 \in T  \right\}.
\end{equation}
We can define the next level element patch as the following
\begin{equation}
	\mathcal K^1_{z_0} = \left\{ T \in \overline{\mathcal T}_{s,h}; z_{\mathcal K^0_{z_0}} \in T \right\},
\end{equation}
where $ z_{\mathcal K^0_{z_0}}$ is a node that belongs to the triangles of $K^0_{z_0}$.
Through induction, we can set the $i^{th}$ level patch as
\begin{equation}
	\mathcal K^i_{z_0} = \left\{ T \in\overline{\mathcal T}_{s,h}; z_{\mathcal K^{i-1}_{z_0}} \in T \right\},
\end{equation}
where $z_{\mathcal K^{i-1}_{z_0}}$ is a node that belongs to the triangles of $K^{i-1}_{z_0}.$
We define the \textbf{element patch} as 
\begin{equation}
	\mathcal K_{z_0} = \left\{ \mathcal K^i_{z_0}; \textrm{card}(\mathcal K^i_{z_0}) > 6
						\textrm{ and } \textrm{card}(\mathcal K^{i-1}_{z_0}) \leq 6 \right\},
\end{equation}
where, $\textrm{card}(\cdot)$ denotes the number of nodes in a set.
Having defined an element patch, we can now define the local gradient recovery operator
$G^h: V^h_s \rightarrow V^h_s \times V^h_s$. Let $z_0$ and $\mathcal K_{z_0}$ again be 
a node in $\overline{\mathcal T}_{s,h}$ and its associated element patch respectively. Let $Z$ be
 the number of nodes in $\mathcal K_{z_0}$ excluding $z_0$.
In addition, let $z_i$, $i\in [0, Z]$ be the nodes of the element patch and $(x_i, y_i)$ be 
the xy--coordinates of $z_i$ and $(\hat x_i, \hat y_i) = (x_i - x_0, y_i - y_0)$ be the
coordinates centered around $z_0$. Define
\begin{equation}
	A_{z_0} = 
	\left[
	\begin{array}{c}
		\mathbf {\hat x}_0^T \\
		 \mathbf {\hat x}_1^T \\
		\vdots \\
		\mathbf {\hat x}_Z^T
	\end{array}
	\right],\ 
	\mathbf {\hat x}_i = 
	\left[
	\begin{array}{c}
		1 \\
		\hat x_i \\
		\hat y_i \\
		\hat x_i^2 \\
		\hat x_i \hat y_i\\
		\hat y_i^2
	\end{array}
	\right],\ 
	\mathbf c_{z_0} = 
	\left[
	\begin{array}{c}
		c_0 \\
		c_1 \\
		\vdots \\
		c_5
	\end{array}
	\right], \ 
	\textrm{ and } 
	\mathbf d_{z_0} = 
	\left[
	\begin{array}{c}
		v^h(z_0) \\
		v^h(z_1) \\
		\vdots \\
		v^h(z_Z)
	\end{array}
	\right].
\end{equation}
The discrete quadratic least--squares fitting of a function
$v^h \in V^h_s$ over the patch $\mathcal K_{z_0}$ can be performed by
solving the least--squares problem
\begin{equation}
	A_{z_0}^T A_{z_0} \mathbf c_{z_0} = A_{z_0}^T \mathbf d_{z_0}
\end{equation}
for the coefficients $\mathbf c_{z_0}$ of the approximating polynomial
and setting
\begin{equation}
	p_{z_0}(x,y) := \mathbf c_{z_0}^T \mathbf P,
\end{equation}
where $\mathbf P = \left[ 1, x-x_0, y-y_0, (x - x_0)^2, (x-x_0)(y-y_0), (y-y_0)^2\right]^T$. 
Now, we define $G^h_{z_0}$, the local gradient recovery operator defined at node $z_0$
as 
\begin{equation}
	G^h_{z_0} u^h(z_0) := \nabla p_{z_0}(x_0, y_0).
\end{equation}
The global gradient recovery operator $G^h$ is a composite of all the local gradient recovery
operators; for each node $z_0 \in \mathcal N$, set the nodal degree of freedom of 
$V^h_s \times V^h_s$ over $\overline {\mathcal T}_{s,h}$ to $G^h_{z_0} v^h(z_0)$.
\begin{figure}
	\centering
	\resizebox{0.1\textheight}{!}{
	\input{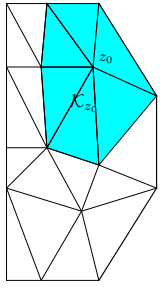}
	}
	\caption{An illustration of the element patch of a node $z_0 \in \overline{\mathcal T}_{s,h}$.}
\end{figure}
\begin{remark}
Computational implementation of $G^h$ is relatively simple due to the fact that
$\nabla p_{z_0}(x_0, y_0) = [c_1, c_2]^T$ at every node $z_0\in \mathcal N$.
\end{remark}
\nocite{*}
\bibliography{bibliography}

\begin{thebibliography}{10}

\bibitem{bernardi1994new}
Christine Bernardi.
\newblock A new nonconforming approach to domain decomposition: the mortar
  element method.
\newblock {\em Nonliner Partial Differential Equations and Their Applications},
  1994.

\bibitem{cockburnextension}
Bernardo Cockburn and Manuel Solano.
\newblock Solving dirichlet boundary--value problems on curved domains by
  extensions from subdomains.
\newblock {\em SIAM Journal on Scientific Computing}, 2012.

\bibitem{day2008analysis}
David Day and Pavel Bochev.
\newblock Analysis and computation of a least-squares method for consistent
  mesh tying.
\newblock {\em Journal of Computational and Applied Mathematics},
  218(1):21--33, 2008.

\bibitem{de2008comparison}
Aukje de~Boer, Alexander~H van Zuijlen, and Hester Bijl.
\newblock Comparison of conservative and consistent approaches for the coupling
  of non-matching meshes.
\newblock {\em Computer Methods in Applied Mechanics and Engineering},
  197(49):4284--4297, 2008.

\bibitem{hecht2005freefem++}
Fr{\'e}d{\'e}ric Hecht, Olivier Pironneau, A~Le~Hyaric, and K~Ohtsuka.
\newblock Freefem++ manual, 2005.

\bibitem{laursen2003consistent}
TA~Laursen and MW~Heinstein.
\newblock Consistent mesh tying methods for topologically distinct discretized
  surfaces in non-linear solid mechanics.
\newblock {\em International journal for numerical methods in engineering},
  57(9):1197--1242, 2003.

\bibitem{levine1985superconvergent}
Nick Levine.
\newblock Superconvergent recovery of the gradient from piecewise linear
  finite-element approximations.
\newblock {\em IMA Journal of numerical analysis}, 5(4):407--427, 1985.

\bibitem{naga2004posteriori}
Ahmed Naga and Zhimin Zhang.
\newblock A posteriori error estimates based on the polynomial preserving
  recovery.
\newblock {\em SIAM Journal on Numerical Analysis}, 42(4):1780--1800, 2004.

\bibitem{parks2007novel}
ML~Parks, L~Romero, and P~Bochev.
\newblock A novel lagrange-multiplier based method for consistent mesh tying.
\newblock {\em Computer methods in applied mechanics and engineering},
  196(35):3335--3347, 2007.

\bibitem{quarteroni1999domain}
Alfio Quarteroni and Alberto Valli.
\newblock {\em Domain decomposition methods for partial differential
  equations}.
\newblock Number CMCS-BOOK-2009-019. Oxford University Press, 1999.

\bibitem{toselli2005domain}
Andrea Toselli and Olof~B Widlund.
\newblock {\em Domain decomposition methods: algorithms and theory}, volume~34.
\newblock Springer, 2005.

\bibitem{wohlmuth2000mortar}
Barbara~I Wohlmuth.
\newblock A mortar finite element method using dual spaces for the lagrange
  multiplier.
\newblock {\em SIAM journal on numerical analysis}, 38(3):989--1012, 2000.

\bibitem{zhang2005new}
Zhimin Zhang and Ahmed Naga.
\newblock A new finite element gradient recovery method: Superconvergence
  property.
\newblock {\em SIAM Journal on Scientific Computing}, 26(4):1192--1213, 2005.

\bibitem{zienkiewiczsuperconvergent}
OC~Zienkiewicz and JZ~Zhu.
\newblock The superconvergent patch recovery (spr) and adaptive finite element
  refinement.
\newblock {\em Computer Methods in Applied Mechanics and Engineering},
  101(1):207--224, 1992.

\bibitem{zienkiewicz1992superconvergent}
Olgierd~Cecil Zienkiewicz and Jian~Zhong Zhu.
\newblock The superconvergent patch recovery and a posteriori error estimates.
  part 1: The recovery technique.
\newblock {\em International Journal for Numerical Methods in Engineering},
  33(7):1331--1364, 1992.

\end{thebibliography}
\end{document}